\documentclass{article}
\usepackage[dvips]{graphics}
\usepackage{amsthm,amsfonts,amsmath, amssymb}
\usepackage{alltt,makeidx,enumerate}
\usepackage{epsfig}
\usepackage{subfig}
\usepackage{sidecap}
\usepackage{color}

\newtheorem{lem}{Lemma}[section]
\newtheorem{thm}[lem]{Theorem}
\newtheorem{prop}[lem]{Proposition}
\newtheorem{coro}[lem]{Corollary}

\theoremstyle{definition}
\newtheorem{example}[lem]{Example}

\newcommand{\me}{\mathrm{e}}
\newcommand{\dif}{\mathrm{d}}

\newcommand{\Div}{\mathrm{div\,}}
\newcommand{\limite}[2]{\lim_{#1\rightarrow #2}}
\newcommand{\deru}[1]{#1^\prime}
\newcommand{\derd}[1]{#1^{\prime\prime}}
\newcommand{\norm}[1]{\left\Vert#1\right\Vert}
\newcommand{\abs}[1]{\left\vert#1\right\vert}
\newcommand{\set}[1]{\left\{#1\right\}}
\newcommand{\mc}[1]{\mathcal{#1}}
\newcommand{\Real}{\mathbb R}
\newcommand{\eps}{\varepsilon}
\newcommand{\To}{\rightarrow}
\newcommand{\B}{\mathcal{B}}
\newcommand{\ts}[1]{\mathbf{#1}}

\newcommand{\pd}[2]{\frac{\partial#1}{\partial#2}}
\newcommand{\td}[2]{\frac{\dif#1}{\dif#2}}
\newcommand{\bu}{{\mathbf{u}}}
\newcommand{\vs}{\left[\frac{r(R)}{R}\right]}

\begin{document}


\title{Infinite energy cavitating solutions: a variational approach}

\author{Pablo V. Negr\'on--Marrero\thanks{pablo.negron1@upr.edu}\\
Department of Mathematics\\
University of Puerto Rico\\
Humacao, PR 00791-4300\\\and Jeyabal
Sivaloganathan\thanks{J.Sivaloganathan@bath.ac.uk}\\
Department of Mathematical Sciences\\University of Bath, Bath\\BA2
7AY, UK }

\date{}

\maketitle

\begin{abstract}
We study the phenomenon of cavitation for the displacement 
boundary-value problem of radial, isotropic compressible elasticity
for a class of stored energy functions of the form $W(F) + h(\det F)$, where  
$W$ grows like
$\norm{F}^n$,  and $n$
is the space dimension. In this case it follows (from a result of Vodop’yanov,  
Gol’dshtein and  Reshetnyak) that discontinuous deformations 
must have infinite energy. After characterizing the rate at which this energy 
blows up, we introduce a
modified energy functional which differs from the original by a null 
lagrangian, 
and for which cavitating energy  minimizers with finite energy exist. In 
particular, the Euler--Lagrange equations for the modified energy
functional are identical to those for the original problem except for
the boundary condition at the inner cavity. This new boundary condition 
states that a certain modified radial Cauchy stress function has to vanish at 
the inner 
cavity. This condition corresponds
to the radial Cauchy stress for the original functional diverging to $-\infty$ 
at
the cavity surface. Many previously known variational  results
for finite energy cavitating solutions now follow for the modified
functional, such as the existence of radial energy minimizers, satisfaction of 
the Euler-Lagrange equations for
such minimizers, and the existence of a critical boundary displacement for
cavitation. We also discuss a numerical scheme for computing these singular 
cavitating
solutions using regular solutions for punctured balls. We show the convergence 
of this
numerical scheme and give some numerical examples including one for the 
incompressible limit case. Our approach is motivated in 
part by the use of the ``renormalized energy" for Ginzberg-Landau vortices.

\end{abstract}
  


\section{Introduction}
Cavitation (i.e., the formation of holes) is a commonly observed
phenomenon in the fracture of polymers and metals (see \cite{GeLi58}).
In his seminal
paper \cite{Ba82}, Ball formulated a variational problem, in the
setting of nonlinear elasticity, for which the energy minimising
radial deformations of (an initially solid) ball
formed a cavity at the centre of the deformed ball when the imposed
boundary loads or displacements were sufficiently large. Following
this paper there have been numerous studies of aspects of the problem
of radial cavitation: some on analytical properties (see, e.g.,
\cite{St85}, \cite{Si86a}, \cite{Me}) and others relating to specific
stored energies (a helpful overview is contained in \cite{HP}).
Subsequent studies, e.g., of \cite{MuSp95}, \cite{SiSp2000a},
\cite{JSp2}, \cite{HMC} have addressed general analytic questions of
existence of cavitating energy minimisers in the non--symmetric case. In all of 
these works, the Dirichlet part of the stored energy function grows like 
$\norm{\nabla\ts{u}}^p$ with $n-1<p<n$, where $\ts{u}$ is a deformation and $n$ 
is the space dimension. The case $p=n-1$ for non cavitating deformations and 
for a three dimensional compressible neo--Hookean material (\cite{PeGu2015}), 
has been studied in \cite{HeRo2018} for axisymmetric bodies. 

In this paper we study radial solutions of the equations of elasticity 
for a spherically symmetric, isotropic, hyperelastic, compressible body, for 
the critical exponent $p=n$. It follows in this case that cavitating solutions 
for the corresponding Euler--Lagrange equations have infinite energy.  Using a 
variational approach, we show that for a general class of 
stored energy functions, the radial equilibrium equations do have cavitating 
solutions with infinite Cauchy stress at the origin and satisfying the outer 
displacement boundary condition. Moreover these solutions are 
characterized as finite energy minimizers of a modified energy functional (cf. 
\eqref{modsef1}) with the same equilibrium equations as the original 
functional. Our approach has connections with work of Henao and 
Serfaty \cite{HeSe2013} and Ca\~nulef-Aguilar and Henao \cite{CaDu2019} for 
incompressible materials and  with the use of the "renormalised" energy in the 
Ginzberg-Landau vortices problem \cite{BeBrHe1994}.

The case $n=2$ of this problem, which corresponds to a two dimensional 
compressible neo--Hookean 
material, was studied by Ball \cite[pp. 606-607]{Ba82} where he proved, for a 
particular stored energy function having logarithmic growth for small 
determinants, the existence of cavitating solutions of the equilibrium 
equations having infinite Cauchy stress at the origin. His approach is based on 
an application of Schauder's fixed point theorem, and although he did not solve 
the full boundary value problem (there was no attempt to match the outer 
boundary condition), the cavity size appears as a parameter in his argument 
which in principle could be adjusted to match the outer boundary condition.
The class of stored energy functions studied in this paper (cf. 
\eqref{ogden-mat}) includes compressible neo-Hookean stored energies widely 
used in applications. The results of this paper, in the case $n=2$, thus allow 
for a variational treatment of cavitation of a disc in two dimensions, which 
has 
not been previously possible for such neo--Hookean stored energy functions. The 
approach should also extend to treat axisymmetric cavitation of a cylinder in 
three dimensions (the work in \cite{HeRo2018}  on 
axisymmetric deformations may be relevant here).

Consider a body which in its reference configuration occupies
the region
\begin{equation}\label{eqn2.1}
\B=\{ \ts{x}\in \Real ^n~:~ \norm{\ts{x}}<1\},
\end{equation}
where $n=2,3$  and $\norm{\cdot}$  denotes the Euclidean norm. Let
$\ts{u}:\B\rightarrow\Real^n$ denote a deformation of the body and
let its \textit{deformation gradient} be
\begin{equation}\label{eqn2.2}
\nabla\ts{u}(\ts{x})=\td{\ts{u}}{\ts{x}}(\ts{x}).
\end{equation}
For smooth deformations, the requirement that $\ts{u}(\ts{x})$ is
locally \textit{invertible and preserves orientation} takes the
form
\begin{equation}\label{eqn2.3}
\det \nabla\ts{u}(\ts{x})>0,\quad\ts{x}\in\B.
\end{equation}
Let $W:M_+^{n\times n}\rightarrow\Real$ be the \textit{stored
energy function} of the material of the body where $M_+^{n\times
n}=\{ \ts{F}\in M^{n\times n}~:~\det \ts{F}>0\}$ and $M^{n\times
n}$ denotes the space of real $n\times n$ matrices. 
We assume that the
stored energy function $W$ satisfies $W\rightarrow\infty$ as
either $\det \ts{F}\rightarrow 0^+$ or
$\norm{\ts{F}}\rightarrow\infty$. The total energy stored in the
body due to the deformation $\ts{u}$ is given by
\begin{equation}\label{eqn2.4}
E(\ts{u})=\int_\B W(\nabla\ts{u}(\ts{x}))\,\dif\ts{x}.
\end{equation}
We consider the problem of determining a 
configuration of the body that satisfies~(\ref{eqn2.3}) almost everywhere and
minimizes~(\ref{eqn2.4}) among all functions
satisfying the boundary condition:
\begin{equation}\label{bcond}
\ts{u}(\ts{x})=\lambda\ts{x},\quad\ts{x}\in\partial\B,
\end{equation}
where $\lambda>0$ is given. Formally, a sufficiently smooth minimizer satisfies 
the equilibrium equations
\begin{equation}\label{EL}
 \Div\left[\td{W}{\ts{F}}(\nabla\ts{u}(\ts{x}))\right]=\ts{0}.
\end{equation}
Note that if the stored energy $W$ satisfies a growth condition of the form
\begin{equation}\label{growth}
c_1\norm{\ts{F}}^n + c_2 \leq W(\ts{F}),\quad\forall\, \ts{F}\mbox{ with 
}\det\ts{F}>0,
\end{equation}
then (cf. \cite{VoGoYu1979}) any discontinuous deformation $\bu$ of 
$\B$ with $\det\nabla\bu>0$ a.e., must have infinite energy. 

For later reference we mention that if $\ts{u}$ is a smooth solution of 
\eqref{EL}, then (see \cite{Green1973})
\begin{equation}\label{Green_div} 
\Div\left[W(\nabla\ts{u})\ts{x}+\left[\pd{W}{\ts{F}}(\nabla\ts{u})\right]
^T\!\!(\ts{u}-(\nabla\ts{u})\ts{x})\right]=nW(\nabla\ts{u}).
\end{equation}
If $\ts{u}$ is smooth except at the origin where it opens up a cavity, 
and 
$\B_\eps$ is a ball of radius 
$\eps>0$ around the origin, then integrating this equation over the punctured 
ball 
$\B\setminus\B_\eps$, we get 
that
\begin{eqnarray}
 n\int_{\B\setminus\B_\eps} 
W(\nabla\ts{u}(\ts{x}))\,\dif\ts{x}&=&\int_{\partial\B}\left[W(\nabla\ts{u})\ts{
x}+\left[\pd{W}{\ts{F}}(\nabla\ts{u})\right]^T\!\! 
(\ts{u}-(\nabla\ts{u})\ts{x})\right]\cdot{\ts{N}}\,\dif s(\ts{x})\nonumber\\
&&-\int_{\partial\B_\eps}\left[W(\nabla\ts{u})\ts{
x}+\left[\pd{W}{\ts{F}}(\nabla\ts{u})\right]^T\!\! 
(\ts{u}-(\nabla\ts{u})\ts{x})\right]\cdot{\ts{N}}\,\dif 
s(\ts{x}),\label{Green_energy}
\end{eqnarray}
where $\ts{N}$ is the outer normal to each boundary. Thus the blow up in the 
energy as $\eps$ becomes small, comes from the integral over the boundary 
$\partial\B_\eps$. Note that this integral is the sum of two terms:
\begin{equation}\label{singBT}
\int_{\partial\B_\eps}\left[W(\nabla\ts{u})\ts{I}-\left[\pd{W}{\ts{F}}(\nabla\ts
{ u }
)\right]^T\!\! 
(\nabla\ts{u})\right]\ts{x}\cdot{\ts{N}}\,\dif 
s(\ts{x}),\,\, 
\int_{\partial\B_\eps}\left[\pd{W}{\ts{F}}(\nabla\ts{u})\right]^T\!\! 
\ts{u}\cdot{\ts{N}}\,\dif 
s(\ts{x}),
\end{equation}
the second one representing as $\eps\To0$, the work done in opening the 
singularity. The tensor in brackets in the first boundary integral above is the 
Eshelby energy--momentum tensor (cf. \cite{Es1951}, \cite{GoMi1996}). It is 
interesting to note 
that if the stored energy function grows like $\norm{\nabla\ts{u}}^p$, then for 
$p<n$ both terms in \eqref{singBT} tend to zero as $\eps\To0$ (cf. 
\cite{SiSp2000b}), while both tend to infinity if $p>n$. In the 
case $p=n$ and in the radial case, we will show that the first term has a 
finite limit while the second one is unbounded as $\eps\To0$.

If the material is homogeneous and $W$ is isotropic and frame
indifferent, then it follows that
\begin{equation}\label{eqn2.5}
W(\ts{F})=\Phi(v_1,\ldots,v_n),\quad\ts{F}\in M_+^{n\times n},
\end{equation}
for some function $\Phi:\Real_+^n\To\Real$ symmetric in its
arguments, where $v_1,\ldots,v_n$ are the eigenvalues of
$(\ts{F}^t\ts{F})^{1/2}$ known as the \textit{principal stretches}.  

We now restrict attention to the special case in which the deformation
$\ts{u}(\cdot)$ is \textit{radially symmetric}, so that
\begin{equation}\label{eqn2.8}
\ts{u}(\ts{x})=r(R)\,\frac{\ts{x}}{R},\quad\ts{x}\in \B,
\end{equation}
for some scalar function $r(\cdot)$, where $R=\norm{\ts{x}}$. In this
case one can easily check that
\begin{equation}\label{eqn2.9}
v_1=\deru{r}(R)~~,~~v_2=\cdots=v_n=\frac{r(R)}{R}.
\end{equation}
Thus (\ref{eqn2.4}) reduces to
\begin{equation}\label{eqn2.10}
E(\ts{u})=\omega_nI(r)=\omega_n\int_0^1 R^{n-1}\Phi \left(
\deru{r}(R),\frac{r(R)}{R},\ldots,\frac{r(R)}{R}\right)\,\dif R,
\end{equation}
where $\omega_n=2\pi $ or $\omega_n=4\pi$ if $n=2\ \mathrm{or} \ 3$
respectively. (In general $\omega_n$ is area of the unit sphere in
$\Real ^n$.)

In accord with (\ref{eqn2.3}) we have the inequalities
\begin{equation}\label{eqn2.11}
\deru{r}(R),\,\frac{r(R)}{R}>0,\quad 0<R<1,
\end{equation}
and \eqref{bcond} reduces to:
\begin{equation}\label{eqn2.12}
r(1)=\lambda.
\end{equation}
Our problem now is to minimize the functional $I(\cdot)$ over the set
\begin{equation}\label{alambda}
\mathcal{A}_\lambda=\set{r\in
W^{1,1}(0,1)\,:\,r(1)=\lambda,\,\deru{r}(R)>0\mbox{
for~a.e.~}R\in(0,1),\, r(0)\ge0}.
\end{equation}
Formally, the Euler--Lagrange equation for $I(\cdot)$ is given by
\begin{equation}\label{eqn2.13}
\td{}{R}\left[R^{n-1}\Phi_{,1}(r(R))\right]=(n-1)R^{n-2}\Phi_{,2}(r(R)),
\quad0<R<1,
\end{equation}
subject to~(\ref{eqn2.12}) and $r(0)\ge0$, where:
\begin{equation}\label{Phi-notation}
\Phi_{,i}(r(R))=\Phi_{,i}\left(\deru{r}(R),\displaystyle\frac{r(R)}{R},\ldots,
\frac{r(R)}{R}\right),\ \ \ i=1,...,n .
\end{equation}

If $c=r(0)>0$, then the deformed ball contains a spherical cavity of radius 
$c$. 
In the case $n=2$,  Ball \cite[pp. 606-607]{Ba82} gives an example of a stored 
energy function satisfying \eqref{growth} and proves existence of corresponding 
radial cavitating equilibrium solutions of \eqref{eqn2.13} which (necessarily) 
have infinite energy. His approach is based on an application of Schauder's 
fixed point
theorem, and although he does not solve the full boundary-value problem (there
was no attempt to match the outer boundary condition), the cavity size appears
as a parameter in his argument which, in principle, could be adjusted to match 
the
outer boundary condition. In this paper we give a characterization of 
cavitating 
equilibria with \emph{infinite energy}   as minimizers of a \emph{modified} 
energy functional, which
is related to the growth of the radial component of the Cauchy stress of an 
equilibrium solution near a point of cavitation.

To 
highlight some of the general structure of the underlying problem, we will 
state 
certain of our results for stored energy functions of the form\footnote{ 
Our results can be readily extended to more general stored energies, e.g.,  of 
the form 
$
\frac{\kappa}{n} \sum_{i=1}^nv_i^n+\psi 
(v_1,...,v_n)+h(v_1v_2\cdots
	v_n)$, under suitable assumptions on $\psi$.}
\begin{equation}\label{ogden-mat}
\Phi(v_1,\ldots,v_n)=\frac{\kappa}{n} \sum_{i=1}^nv_i^n+h(v_1v_2\cdots
v_n),
\end{equation}
where $\kappa >0$ and $h\,:\,(0,\infty)\To[0,\infty)$ is a $C^1$ function that
satisfies
\begin{subequations}\label{hcond}
\begin{eqnarray}
 \derd{h}(d)>0,&~&\forall\,d>0,\label{hcond1}\\
\limite{d}{0^+}h(d)=\infty,&\quad&
\limite{d}{\infty}\frac{h(d)}{d}=\infty,\label{hcond2}\\
\limite{d}{0^+}\deru{h}(d)=-\infty,&\quad&
\limite{d}{\infty}\deru{h}(d)=\infty.\label{hcond3}
\end{eqnarray}
\end{subequations}
In this case, the energy 
functional $I(r)$ in \eqref{eqn2.10} takes the form
\begin{equation}\label{originalsef}
I(r)=\int_0^1 R^{n-1}\left[\frac{\kappa}n \left((r'(R))^n +(n-1) 
\left(\frac{r(R)}{R}\right)^n\right) +h(\delta (R))\right] \, dR,
\end{equation} 
where
\[
  \delta(R) = r'(R) \left(\frac{r(R)}{R}\right)^{n-1}.
\]
It is clear that (discontinuous) radial deformations with $r(0)>0$ must have 
infinite energy as a result of the second bracketed term in the integrand. In 
Section \ref{sec:modfunc}, in the spirit of the ``renormalized energy 
approach'' 
for Ginzberg-Landau  vortices (see, e.g., \cite{BeBrHe1994}),  we characterize 
the order of the singularity in the energy and in the radial component of 
Cauchy 
stress for a cavitating solution as logarithmic in all dimensions. To motivate 
the form of the regularisation, we use the specialization of 
\eqref{Green_div} 
to the radial case satisfied by smooth solutions of the radial equilibrium 
equation \eqref{eqn2.13}:
\begin{equation}\label{div_ident}
R^{n-1}\Phi(r(R)) = 
\frac{d}{dR}\left[\frac{R^n}{n}\left(\Phi(r(R)) 
-r'\Phi_{,1}(r(R))\right) +\frac{r^n}{n}T(r(R))\right],
\end{equation} 
with the notation in \eqref{Phi-notation} and where 
\begin{equation}\label{Cauchy_stress}
T(r(R))=\left[\frac{R}{r(R)}\right]^{n-1}\Phi_{,1}(r(R)),
\end{equation}
is the radial  
component of the Cauchy stress. Integrating the above identity from $R=\eps$ 
to $R=1$ for a cavitating solution and using the form of the stored energy 
function \eqref{ogden-mat}, we show that all boundary terms have a 
finite limit as $\epsilon 
\rightarrow 0$ apart from the term
\begin{equation}\label{WorkT}
-\lim_{\epsilon\rightarrow 0} 
\frac{r(\epsilon)^n}{n}\,T(r(\epsilon)),
\end{equation}
which corresponds to the second term in \eqref{singBT}. Thus, the 
infinite energy of a 
radial solution of the equilibrium equation with 
$r(0)>0$  corresponds to a singularity in the radial Cauchy stress. Thus the 
term \eqref{WorkT} 
can be formally interpreted as the (infinite) 
work required to open the cavity. (If $r(0)=0$, then this term is zero.) Thus 
for a cavitating solution
\[
\lim_{\eps\rightarrow 0}\left[\int_\eps^1 R^{n-1} \Phi\left(r'(R), 
\frac{r(R)}{R},\ldots, 
\frac{r(R)}{R}\right)\, dR+
\frac{r(\epsilon)^n}{n}\,T(r(\epsilon))\right],
\]
is finite.
Using the characterization of the asymptotic behaviour of the  
Cauchy stress given in Proposition \ref{Tasint}, we introduce a modified
energy functional, given by
\begin{eqnarray}
\hat I(r)&=&\int_0^1 R^{n-1} \Phi\left(r'(R), \frac{r(R)}{R},\ldots, 
\frac{r(R)}{R}\right)\, dR\nonumber\\&& -\dfrac{\kappa(n-1)}{n} 
\lim_{R\rightarrow 0}r^n(R)\ln 
\left(\frac{r(R)}{R}\right),\label{modsef1b}
\end{eqnarray}
where the last term accounts for the singular behaviour in \eqref{WorkT}.
This functional can also be expressed as
\begin{eqnarray*}
\hat I(r)&=&\int_0^1  
R^{n-1}\left[\frac{\kappa}{n}\left(r'\right)^n 
+h\left(\delta (R)\right)
+\kappa(n-1)\delta(R)\left(\dfrac{1}{n}+ \ln\left(\frac{r}{R}\right)\right)
\right]\,\dif R\nonumber\\
&&~~~~~~~~~~~~~~~~~~~~~~- \frac{\kappa(n-1)}{n}\lambda^n\ln\lambda.
\end{eqnarray*}
It is easy to now show that there are $r\in\mathcal{A}_\lambda$, with $r(0)>0$  
for which $\hat I(r)$ 
is finite. Moreover, the 
Euler-Lagrange equation for this modified functional coincides with 
that for the original functional \eqref{originalsef} since, by construction, 
they differ by a null lagrangian term (see Theorem \ref{thm:3}). Moreover, for 
many deformations 
with $r(0)=0$, in particular for the homogeneous deformation $r(R)\equiv 
\lambda R$, the two energies coincide. However, we will show that for $\lambda$ 
sufficiently large, energy minimisers for the modified functional must satisfy 
$r(0)>0$.

Many known results for finite energy cavitating solutions 
(see, e.g., \cite{Ba82},  \cite{St85}, \cite{Si86a}) now follow by similar 
methods for
the modified functional \eqref{modsef1b}. In particular in Section 
\ref{sec:exist-EL} we show
that minimizers of the modified functional exist and that they satisfy the
corresponding Euler-Lagrange equations for such minimizers.  Moreover, in 
Theorem \ref{lambda_crit} we show
the existence of a critical boundary displacement $\lambda_c$ for
cavitation for the modified functional, below of which the minimizers of the
modified functional must be homogeneous.

In Section \ref{sec:punct} we discuss a numerical scheme for
computing the cavitating solutions of the modified functional via solutions
on punctured balls. In the usual cavitation problem, the convergence of the
solutions on these punctured balls to a solution on the full ball follows from
the corresponding properties of solutions of the Euler-Lagrange equations and by 
a phase
plane analysis (cf. \cite{Si86a}). Since the Euler-Lagrange equations for our 
modified
functional are equal (except for the boundary condition at the inner cavity) to
those of the original functional, the proof of
convergence of the punctured ball solutions in the case of the modified
functional is essentially the same as that for a functional in which we have 
$\frac{\kappa}{p} \sum_{i=1}^nv_i^p+h(v_1v_2\cdots
v_n)$ with $p<n$, instead of \eqref{ogden-mat}. In this
section we also discuss some aspects of the convergence of the corresponding
strains of the punctured ball solutions depending on the size of the boundary
displacement. Finally we close with some numerical examples in Section
\ref{sec:num} which includes one for the incompressible limit case.

\section{The modified energy functional}\label{sec:modfunc}
We call any solution of \eqref{eqn2.13} for which $r(0)>0$ a 
\textit{cavitating solution}. In this section we introduce a modified functional
$\hat{I}(\cdot)$ defined over $\mathcal{A}_\lambda$, having the same 
Euler-Lagrange equation as $I(\cdot)$, for which cavitating solutions have 
finite modified energy, and for which the corresponding modified radial Cauchy 
stress function is increasing on cavitating solutions.  To achieve this, we 
first \emph{assume} the existence of a cavitating solution and obtain 
corresponding estimates that help us to better understand the rate at which the 
energy of a cavitating solution and the corresponding radial Cauchy stress blow 
up at the origin. We then use these estimates to construct a modified 
variational problem, using which we are able to prove \emph{a posteori} that 
such 
solutions exist.

Some of the results in this section are stated for general stored energy 
functions satisfying the following conditions:\\
\begin{itemize}
 \item[H1:]
$ \Phi_{,11}(q,v,\ldots,v)>0,\quad\forall\,q,v>0;$\\
\item[H2:]
$\displaystyle\frac{\Phi_{,1}(q,v,\ldots,v)-\Phi_{,2}(q,v,\ldots,v)}{q-v} 
+\Phi_{,12}(q,v,\ldots,v)\geq 0,\quad\forall\,q,v>0,\quad q\ne v$;\\ 
\item[H3:]
$\displaystyle
R(q,v)\equiv\frac{q\Phi_{,1}(q,v,..v)-v\Phi_{,2}(q,v,..,v)}{q-v}>0,\quad\forall 
q\ne v$;\\
\item[H4:]
$\displaystyle\frac{\partial R(q,v)}{\partial q}\geq 0$ 
for $0< q\leq v$.\\
\end{itemize}
Is easy to check that \eqref{ogden-mat} satisfy these conditions as well.

We shall make use of the following well known properties of solutions of 
\eqref{eqn2.13} (cf. \cite{Ba82}, \cite{St85}).
\begin{prop}\label{well_known_properties}
Let $r\in C^2((0,1])\cap C([0,1])$ be a solution of \eqref{eqn2.13} on $[0,1]$ 
satisfying $r(0)>0$ and such that 
$\delta (R) := r'(R)\left(\frac{r(R)}{R}\right)^{n-1}$ is bounded on $[0,1]$. 
Then
\begin{enumerate}
\item
$r'(R)<\frac{r(R)}{R}$ on $(0,1]$,
\item
$r'(R) \rightarrow 0$ and $\frac{r(R)}{R}\rightarrow \infty$ as 
$R\rightarrow 0$.
\item
Let $\Phi$ satisfy (H1) and (H2). 
Then any cavitation solution of \eqref{eqn2.13} 
satisfies $\derd{r}(R)\geq 0$.
\end{enumerate}
\end{prop}

\subsubsection*{Asymptotic behaviour of the Radial Cauchy Stress}
The Cauchy stress \eqref{Cauchy_stress} corresponding to a solution of the 
radial equilibrium equation \eqref{eqn2.13} satisfies
\begin{equation}\label{TodeR}
\frac{d}{dR}\, 
T(r(R))=\frac{(n-1)R^{n-1}}{r^{n}(R)}\left(\frac{r(R)}{R}\Phi_{,2}(r(R))
-r'\Phi_{,1}(r(R))\right).
\end{equation}
For later use, we invert the relation 
$T=\frac{\Phi_1\left(v_1,v_,..,v\right)}{v^{n-1}}$ to obtain $v_1=g(v,T)$  and 
then rewrite \eqref{TodeR} in terms of the independent variable 
$v=\frac{r}{R}$ as
\begin{equation}\label{Todev}
\frac{dT(v)}{dv}= -\frac{(n-1)}{v^{n}}\left(\frac{{v{\Phi_{,2} 
(g(v,T),v,..v)-g(v,T){\Phi_{,1}(g(v,T),v,..,v)}}}}{v-g(v,T)}\right).
\end{equation}
It follows from \eqref{TodeR}, 
\eqref{Todev}, and (H3) that $T(r(\cdot))$ is monotonic as a 
function of $R$ or $v$ along radial solutions.

For the specific class of stored energy functions \eqref{ogden-mat}, equation 
\eqref{TodeR} becomes
\[
\frac{d}{dR} T(r(R)) =\frac{(n-1)\kappa}{R} 
-\frac{(n-1)R^{n-1}\kappa}{r^{n}}(r')^n .
\]
The second term on the right hand side is integrable on $[0,1]$ for a 
cavitating 
solution $r$ and so 
\[
T(r(R)) = {(n-1)\kappa}\ln(R) + O(1) \ \mathrm{as} \ R\rightarrow 0.
\]
In addition, for the stored energy function \eqref{ogden-mat}, equation 
\eqref{Todev} reduces to
\begin{eqnarray*}
\frac{dT(v)}{dv}&=& 
-\frac{(n-1)}{v^{n}}\kappa\left(\frac{v^n-g(v,T)^n}{v-g(v,T)}\right)\\&=& 
-\frac{\kappa(n-1)}{v}\left(1 +\frac{g}{v}+...+\frac{ g^{n-1}}{v^{n-1}}\right).
\end{eqnarray*}
Now integrating on $[\lambda, v]$ yields
\begin{eqnarray*}
T(v)+\kappa(n-1) \ln v &=& T(\lambda) +\kappa(n-1) \ln\lambda\\
&&-\kappa(n-1)\int_{\lambda}^{v}\left(\frac{g}{w^2}+\ldots+ 
\frac{g^{n-1}}{w^n}\right) \, \dif w,
\end{eqnarray*}
showing that the growth in $T(v)$ is logarithmic in the variable $v$ as 
$v\rightarrow \infty$. We summarize these results in the following proposition.
\begin{prop}\label{Tasint}
 Let $r\in C^2((0,1])\cap C([0,1])$ be a solution of \eqref{eqn2.13} on $[0,1]$ 
satisfying $r(0)>0$. Then, for the stored energy function \eqref{ogden-mat}, 
the radial component of Cauchy stress given by \eqref{Cauchy_stress} satisfies
\[
 \limite{R}{0^+}(T(r(R)) - {(n-1)\kappa}\ln(R))
\]
is finite, and as a function of the circumferential strain $v=\frac{r}{R}$,
\[
 \limite{v}{\infty}(T(v)+\kappa(n-1) \ln v)
\]
is finite. In particular, 
$\limite{R}{0^+}T(r(R))=\limite{v}{\infty}T(v)=-\infty$.
\end{prop}

\subsubsection*{Asymptotic behaviour of the determinant} 
\begin{lem}\cite[Theorem 3.1]{St93}\label{detR1}
Assume that (H1)--(H4) hold. Then the determinant function 
$\delta$ (see 
\eqref{originalsef}) corresponding to a cavitation solution, 
as a function of the circumferential strain $v=\frac{r}{R}$, satisfies:
\begin{equation}
\frac{1}{v^{n-1}}\Phi_{,11}\td{\delta}{v}=(n-1)v^{-1}\left(q(v)-v\right)\frac{
\partial R}{\partial q}(q(v),v),
\end{equation}
where $q(v)=\delta(v)/v^{n-1}$. Hence $\delta(v)$ is a monotone decreasing 
function of $v$.
\end{lem}
Combining Proposition \ref{Tasint} and Lemma \ref{detR1} we get now that:
\begin{coro}\label{coro:4}
 For the stored energy function \eqref{ogden-mat}, the determinant function 
corresponding to a cavitation solution, as a function of the circumferential 
strain $v$, satisfies:
\begin{eqnarray}
 &\displaystyle\hspace{-2.25in}
 \left(1+\frac{1}{(n-1)\kappa}\frac{v^{n(n-1)}}{\delta^{n-2}}
\derd{h}(\delta)\right)\td{\delta}{v}=&\nonumber\\
&\hspace{1.25in}\displaystyle
\frac{v^{n(n-2)}}{\delta^{n-2}}
\left[-v^{n-1}-v^{n-1}\sum_{j=1}^{n-2}v^{-j}\left(\frac{\delta}{v^{n-1}} 
\right)^j +(n-1)
\frac{\delta^{n-1}}{v^{(n-1)^2}} \right ] .&\label{detv_ode}
\end{eqnarray}
Moreover, provided $\deru{h}(d)\To-\infty$ as $d\To0^+$, then $\delta(v)\To0^+$ 
as $v\To\infty$.
\end{coro}

\vspace{0.2in}
We close this section now by establishing conditions under which the term 
in the energy functional \eqref{originalsef}, containing the function 
$h(\cdot)$, is finite for a cavitating 
solution.
\begin{prop}\label{finitehint}
Let the function $h(\cdot)$ in \eqref{ogden-mat} satisfy the inequalities
\begin{equation}\label{growthh2}
 \frac{K_1}{d^{\gamma+2}}\le\derd{h}(d)\le\frac{K_2}{d^{\gamma+2}},\quad 
d\le d_0,
\end{equation}
and
\begin{equation}\label{growthh2a}
\frac{K_1}{d^{\gamma}}\le h(d)\le\frac{K_2}{d^{\gamma}},\quad 
d\le
d_0,
\end{equation}
for some $\gamma>0$ and $d_0>0$. Then the determinant function $\delta(\cdot)$ 
corresponding 
to a cavitating solution satisfies that the 
integral $\int_0^1R^{n-1}h(\delta(R))\,\dif R$ is finite.
\end{prop}
\begin{proof}
 Since by Corollary \ref{coro:4}, $\delta(v)\To0$ as $v\To\infty$, we have that 
for some $v_0>0$,
\[
 \frac{\delta(v)}{v^{n-1}}<\frac{v}{2},\quad v\ge v_0,
\]
where $\delta(v_0)\le d_0$. Using this we get that
\begin{eqnarray*}
-v^{n-1}-v^{n-1}\sum_{j=1}^{n-2}v^{-j}\left(\frac{\delta}{v^{n-1}} 
\right)^j+(n-1)
\frac{\delta^{n-1}}{v^{(n-1)^2}}&\ge& 
-v^{n-1}-v^{n-1}\sum_{j=1}^{n-2}v^{-j}\left(\frac{v}{2}\right)^j \\
&=&v^{n-1}(-2+2^{2-n}).
\end{eqnarray*}
Similarly we can get that
\[
 -v^{n-1}-v^{n-1}\sum_{j=1}^{n-2}v^{-j}\left(\frac{\delta}{v^{n-1}} 
\right)^j+(n-1)
\frac{\delta^{n-1}}{v^{(n-1)^2}}\le v^{n-1}(-1+(n-1)2^{1-n}).
\]
It follows now from \eqref{growthh2} that
\[
\frac{\kappa(n-1)}{2K_2v^n}\delta^{\gamma+2}\le\frac{v^{n(n-2)}}{\delta^{n-2}} 
\left(1+\frac{1}{(n-1)\kappa}\frac{v^{n(n-1)}}{\delta^{n-2}}
\derd{h}(\delta)\right)^{-1}\le\frac{\kappa(n-1)}{K_1v^n}\delta^{\gamma+2}.
\]
It follows now from \eqref{detv_ode}, and the previous estimates that
\[
 (n-1)\frac{\kappa}{vK_1}(-2+2^{2-n})\delta^{\gamma+2}\le\td{\delta}{v}\le (n-1)
\frac{\kappa}{2vK_2}(-1+(n-1)2^{1-n})\delta^{\gamma+2},
\quad v\ge v_0.
\]
We have now that for some constants $C_i$, $i=1,2,3,4$, with $C_1$ and $C_3$ 
positive, that
\begin{equation}\label{growthhp}
 C_1\ln(v)+C_2\le\frac{1}{\delta^{\gamma+1}}\le C_3\ln(v)+C_4, \quad v\ge v_0.
\end{equation}
The result now follows from this estimate and the hypothesis 
\eqref{growthh2a}.
\end{proof}

\subsubsection*{Asymptotic behaviour of the energy functional}
We next study the rate at which the stored energy of a cavitating equilibrium  
diverges to infinity for stored energy functions of the form \eqref{ogden-mat}. 
We do this using the divergence identity \eqref{div_ident}.
\begin{prop}
Suppose that $\Phi$ is of the form \eqref{ogden-mat} and define the modified 
energy functional by
\begin{equation}\label{modsef1}
\hat I(r):=\lim_{\epsilon\rightarrow 0^+}\left[
\int_{\epsilon}^{1}R^{n-1}\Phi (r(R))\, \dif R 
-\frac{\kappa (n-1)}{n}\, r(\eps)^n\ln 
\left(\frac{r(\eps)}{\eps}\right) 
\right].
\end{equation}
Assume that the function $h(\cdot)$ in \eqref{ogden-mat} satisfies 
that
\begin{equation}\label{const_H1}
 \abs{d\,\deru{h}(cd)}\le K[h(d)+1],\quad\abs{c-1}\le\gamma_0,
\end{equation}
for some positive constants $K,\gamma_0$. Let $r\in C^2((0,1])\cap 
C([0,1])$ be a solution of \eqref{eqn2.13} on $[0,1]$. Then
\begin{enumerate}
\item
if $r(0)>0$ and $\delta (R) := r'(R)\left(\frac{r(R)}{R}\right)^{n-1}$ is 
bounded on $[0,1]$, the modified 
energy $\hat I(r)$ is finite and given by 
\begin{eqnarray}
&\displaystyle\frac 1n 
\left[\Phi(r(1))-r'(1)\Phi_{,1}(r(1)) +\lambda^nT(\lambda)\right]
-\frac{\kappa(n-1)}{n^2}r(0)^n
~~~~~~~~~~~~~~~~~~~~~~~~~&\nonumber\\
&~~~~~~~~~~~~~~~- \displaystyle
\limite{\eps}{0^+}\left[T(r(\eps))+\kappa(n-1)\ln\left(\frac{r(\eps)}{\eps}
\right)\right]\frac{r^n(\eps)}{n}.&\label{energy_cons}
\end{eqnarray}
\item
If $r(0)=0$, then  $\hat I(r) =I(r)$ (possibly infinite). (So that 
$\hat I$ agrees with the unmodified energy functional on non-cavitating 
equilibria).
\end{enumerate}
\end{prop}
\begin{proof}
By the fundamental Theorem of Calculus, it follows that 
\begin{eqnarray*}
&\displaystyle\int_{\epsilon}^{1}R^{n-1}\Phi (r(R))\, \dif R 
-\frac{\kappa (n-1)}{n}\, r(\eps)^n\ln 
\left(\frac{r(\eps)}{\eps}
\right)~~~~~~~~~~~~~~~~~~~~~~~~~~&\\
&=\displaystyle\int_\eps^1 \left[R^{n-1} \Phi(r(R)) +
\frac{\kappa(n-1)}{n}\frac{d}{dR}\left(
 \ln \left(\frac{r}{R}\right)r^n\right)\right]\, dR- \frac{\kappa 
(n-1)}{n}\lambda^n\ln\lambda.&
\end{eqnarray*}
On noting that 
\[
\frac{d}{dR}\left(
\ln \left(\frac{r}{R}\right)r^n\right) =
nr^{n-1}r'\ln \left(\frac rR\right) +\frac{d}{dR}\left(\frac {r^n}n\right) 
-R^{n-1}\left(\frac rR\right)^n,
\]
it follows from \eqref{ogden-mat} and the above that $\hat{I}$ 
is also expressible as 
\begin{eqnarray*}
\hat I(r)&=&\limite{\eps}{0^+} \int_\eps^1  
R^{n-1}\left[\frac{\kappa}{n}\left(r'\right)^n
+h\left(\delta (R)\right)
+\kappa(n-1)\delta(R) \left(\dfrac{1}{n}+\ln\left(\frac{r}{R}\right)\right)
\right]\,\dif R\nonumber\\
&&~~~~~~~~~~~~~~~~~~~~~~
- \frac{\kappa (n-1)}{n}\lambda^n\ln\lambda.
\end{eqnarray*}
By Proposition \ref{finitehint} and equation \eqref{bddbe} below, the 
integrand in this expression is integrable on $(0,1)$ for a 
cavitating solution. Thus the limit as 
$\eps\To0^+$ is finite and equal to
\begin{eqnarray}
\hat I(r)&=&\int_0^1  
R^{n-1}\left[\frac{\kappa}{n}\left(r'\right)^n 
+h\left(\delta (R)\right)
+\kappa(n-1)\delta(R)\left(\dfrac{1}{n}+ \ln\left(\frac{r}{R}\right)\right)
\right]\,\dif R\nonumber\\
&&~~~~~~~~~~~~~~~~~~~~~~- \frac{\kappa 
(n-1)}{n}\lambda^n\ln\lambda.\label{functIm}
\end{eqnarray}

As $r$ is a solution of \eqref{eqn2.13}, it follows from \eqref{div_ident} that
\begin{eqnarray*}
 \int_\eps^1R^{n-1}\Phi(r(R))\,\dif 
R&=&\frac{1}{n}\left[\Phi(r(1)) 
-\deru{r}(1)\Phi_{,1}(r(1)) +\lambda^nT(r(1))\right]\\
&&-\frac{1}{n}\left[\eps^n\left(\Phi(r(\eps)) 
-\deru{r}(\eps)\Phi_{,1}(r(\eps))\right) +r(\eps)^nT(r(\eps))\right]
\end{eqnarray*}
Since $R^{n-1}h(\delta(R))$ is integrable in $(0,1)$ and $\deru{r}$ is bounded, 
it follows that
\[
\limite{\eps}{0^+}\eps^n\Phi(r(\eps))=\limite{\eps}{0^+}\frac{\kappa(n-1)}{n}
r(\eps)^n.
\]
Similarly, this time using \eqref{const_H1}, we obtain
\[
 \limite{\eps}{0^+}\eps^n\deru{r}(\eps)\Phi_{,1}(r(\eps))=0.
\]
The result \eqref{energy_cons} follows from these limits, definition 
\eqref{modsef1}, and Proposition \ref{Tasint}.

For the second part, let $L=\limite{R}{0^+}\frac{r(R)}{R}$ and assume that 
$\frac{r(R)}{R}$ is not constant. If 
$L\in[0,\infty)$, is easy to show that
\[
 \limite{\eps}{0^+}r(\eps)^n\ln\left(\frac{r(\eps)}{\eps}\right)=0.
\]
Thus in this case $\hat{I}(r)=I(r)$. Assume now that $L=\infty$. By Rolle's 
theorem and the continuity of $\deru{r}$ in $(0,1]$, it follows that  
$L=\limite{j}{\infty}\deru{r}(R_j)$ for some sequence 
$R_j\To0^+$. Since
$\deru{r}(R)<\frac{r(R)}{R}$ in $(0,1]$, we have by 
\cite[Proposition 6.2]{Ba82}, that $T(r(R))$ is strictly increasing. But 
$\limite{j}{\infty}\frac{r(R_j)}{R_j}=\limite{j}{\infty}\deru{r}(R_j)=\infty$ 
implies that 
$T(r(R_j))\To\infty$ as $j\To\infty$ which contradicts that $T(r(R))$ is 
strictly 
increasing. Thus $L<\infty$ which completes the proof of the second part.
\end{proof}
Henceforth we shall employ the representation \eqref{functIm} as that of our 
modified functional. For later reference we observe that
\begin{equation}\label{bddbe}
\int_0^1R^{n-1}\delta(R)\ln\left(\frac{r}{R}\right)\,\dif
R=\int_{r(0)}^\lambda u^{n-1}\ln(u)\,\dif u
-\int_0^1R^{n-1}\delta(R)\ln(R)\,\dif R,
\end{equation}
which implies that \eqref{functIm} is bounded below.

\section{Existence of minimizers and the Euler-Lagrange equations for the 
modified
functional}\label{sec:exist-EL}
In this section we show some of the details of the analysis that establishes
the existence of minimizers for the modified functional \eqref{functIm} over
\eqref{alambda} and their characterization via the Euler-Lagrange equations. 
The analysis
is very similar to that in \cite[Section 7]{Ba82} and thus we just highlight
the details concerning the extra or new terms in \eqref{functIm}. In this
respect we mention that the stored energy function corresponding to the
modified functional \eqref{functIm} is given by \eqref{mod_Phi} and does not
correspond to an isotropic material. Thus the results in \cite{Ba82} do not
necessarily apply immediately.
\begin{thm}\label{min_existence}
 Assume that the function $h(\cdot)$ is a nonnegative convex function
satisfying \eqref{hcond}. Then the functional \eqref{functIm} has a minimizer
over the set \eqref{alambda}.
\end{thm}
\begin{proof}
 Since the homogeneous deformation $r_h(R)=\lambda R$ belongs to
\eqref{alambda} and $\hat{I}(r_h)<\infty$, this together with $\hat{I}$
bounded below shows that
\[
 \inf_{r\in\mathcal{A}_\lambda}\hat{I}(r)\in\Real.
\]
Let $(r_k)$ be an infimizing sequence. As in \cite{Ba82}, we use the 
change of variables $\rho=R^n$, and set $u_k(\rho)=r_k^n(\rho^{1/n})$. It 
follows now 
that
\[
 \dot{u}_k(\rho)=\td{u_k}{\rho}(\rho)=\delta_{k}(\rho^{1/n}),\quad 
\delta_{k}(R)=\deru{r_k}(R)\left(\frac{r_k(R)}{R}\right)^{n-1}.
\]
From the boundedness of $(\hat{I}(r_k))$ we get that the sequence
\[
 \left(\int_0^1h(\dot{u}_k(\rho))\,\dif \rho\right),
\]
is bounded. It follows now from \eqref{hcond2} and De La Vall\'ee-Poussin 
Criterion that for some subsequence (not relabeld) $(\dot{u}_k)$, we 
have $\dot{u}_k\rightharpoonup w$ in $L^1(0,1)$ for some $w\in L^1(0,1)$ with 
$w>0$ a.e. Letting
\[
 u(\rho)=\lambda^n-\int_\rho^1w(s)\,\dif s,
\]
and $r(R)=u(R^n)^{1/n}$, we get now that $r_k\rightharpoonup 
r$ in $W^{1,1}(\eps,1)$ and that 
$\delta_{k}\rightharpoonup\delta=\deru{r}(r/R)^{n-1}$ in $L^1(\eps,1)$ for any 
$\eps\in(0,1)$. Using \eqref{bddbe} we get that
\[
 \int_\eps^1R^{n-1}\delta_k(R)\ln\left(\frac{r_k}{R}\right)\,\dif
R=\int_{r_k(\eps)}^\lambda u^{n-1}\ln(u)\,\dif u
-\int_\eps^1R^{n-1}\delta_k(R)\ln(R)\,\dif R.
\]
Now using that $r_k(\eps)\To r(\eps)$, 
$\delta_k\rightharpoonup\delta$ in $L^1(\eps,1)$,
and that $R^{n-1}\ln(R)$ is bounded on $(\eps,1)$, we have that
\begin{eqnarray*}
\limite{k}{\infty}\int_\eps^1R^{n-1}\delta_k(R)\ln\left(\frac{r_k}{R}\right)\,
\dif
R&=&\int_{r(\eps)}^\lambda u^{n-1}\ln(u)\,\dif u
-\int_\eps^1R^{n-1}\delta(R)\ln(R)\,\dif R,\\
&=&\int_\eps^1R^{n-1}\delta(R)\ln\left(\frac{r}{R}\right)\,\dif R.
\end{eqnarray*}
This together with the convergence of $(r_k)$ and $(\delta_k)$ already
established, and a weak lower semi--continuity argument shows that
\[
\hat{I}_\eps(r)\le\liminf_k\hat{I}_\eps(r_k)
\le\liminf_k\hat{I}(r_k)
=\inf_{r\in\mathcal{A}_\lambda}\hat{
I } (r),
\]
where $\hat{I}_\eps$ is as in \eqref{functIm} but integrating over $(\eps,1)$. 
We get now from the Monotone Convergence Theorem and the arbitrariness of 
$\eps$ that
\[
\hat{I}(r)\le\liminf_k\hat{I}(r_k)=\inf_{r\in\mathcal{A}_\lambda}\hat{
I } (r),
\]
Since $\lambda=r_k(1)\To r(1)$, we get that 
$r\in\mathcal{A}_\lambda$ and is
therefore a minimizer of $\hat{I}$.
\end{proof}

If we define
\begin{equation}\label{mod_Phi}
 \hat{\Phi}(v_1,\ldots,v_n)=\frac{\kappa}{n}\,v_1^n+h(v_1\cdots v_n)
 +\kappa v_1\cdots v_n\left(\frac{n-1}{n}+\ln(v_2\cdots v_n)\right),
\end{equation}
then (cf. \eqref{functIm})
\[
 \hat{I}(r)=\int_0^1R^{n-1}\hat{\Phi}\left(\deru{r}(R),\frac{r(R)}{R},\ldots,
\frac{r(R)}{R}\right)\,\dif R- \frac{\kappa 
(n-1)}{n}\lambda^n\ln\lambda.
\]
Note that $\hat{\Phi}$ does not correspond to an 
isotropic material as it is not symmetric in its arguments. However we still 
have that $\hat{\Phi}_{,k}(q,t,\ldots,t)=\hat{\Phi}_{,j}(q,t,\ldots,t)$ for 
$2\le k,j\le n$, and that $\hat{\Phi}$ satisfy (H1)--(H4).

With 
$\hat{\Phi}_{,i}(r(R))=\hat{\Phi}_{,i}\left(\deru{r}(R),\frac{r(R)}{R},\ldots,
\frac{r(R)}{R}\right)$, $i=1,2$, we have that
\begin{subequations}\label{mod_Phi_par}
 \begin{eqnarray}
  \hat{\Phi}_{,1}(r(R))
&=&\kappa\deru{r}(R)^{n-1}\nonumber\\
&+&\vs^{n-1}\left[\deru{h}(\delta(R))+(n-1)\kappa\left(\frac{1}{n}+
\ln\vs\right)\right],\label{mod_Phi_par1}\\
~~~\hat{\Phi}_{,2}(r(R))
&=&\deru{r}(R)\vs^{n-2}\bigg[\deru{h}(\delta(R))+\kappa+\kappa(n-1)
\left(\frac{1}{n}+\ln\vs\right)\bigg]
, \label{mod_Phi_par2}
\end{eqnarray}
\end{subequations}
and we define
\begin{eqnarray}
 \hat{T}(r(R))&=&\vs^{1-n}\hat{\Phi}_1(r(R)),\nonumber\\
 &=&\kappa R^{n-1}\left[\frac{\deru{r}(R)}{r(R)}\right]^{n-1}
 +\deru{h}(\delta(R))
 +(n-1)\kappa\left(\frac{1}{n}+\ln\vs\right).\label{mod_CS}
\end{eqnarray}
We call $\hat{T}(r(\cdot))$ the \textit{modified radial Cauchy stress}. The 
techniques in \cite{Ba82} can now be adapted to show the following 
result.
\begin{thm}\label{mod_EL1}
 Let $r$ be any minimizer of $\hat{I}$ over \eqref{alambda}. Assume that the
function $h(\cdot)$ satisfies \eqref{const_H1}. Then $r\in
C^1(0,1]$, $\deru{r}(R)>0$ for all $R\in(0,1]$, $R^{n-1}\hat{\Phi}_1(r(R))$ is
$C^1(0,1]$, and
\begin{equation}\label{mod_EL}
\td{}{R}\left[R^{n-1}\hat{\Phi}_1(r(R))\right]=
(n-1)R^{n-2}\hat{\Phi}_2(r(R)).
\end{equation}
Moreover, if $r(0)>0$, then
\begin{equation}\label{mod_BC_0}
 \limite{R}{0^+}\hat{T}(r(R))=0.
\end{equation}
\end{thm}

The next two results are rather straightforward to verify but they will 
be quite important for the rest of our development, especially for the phase 
plane analysis of \eqref{mod_EL}.
\begin{thm}\label{thm:3}
 Let $r$ be any minimizer of $\hat{I}$ over \eqref{alambda} and assume that
\eqref{const_H1} holds. Then $r$ is a solution of \eqref{eqn2.13} where $\Phi$
is as in \eqref{ogden-mat}.
\end{thm}
\begin{proof}
 We know $r\in C^1(0,1]$. Thus we can expand the following term in
\eqref{mod_EL}:
\begin{eqnarray*}
\td{}{R}\left[R^{n-1}\vs^{n-1}\ln\vs\right]&=&r(R)^{n-2}\deru{r}(R)\left[
1+(n-1)\ln\vs\right]\\&&-R^{n-2}\vs^{n-1}.
\end{eqnarray*}
Substituting this into \eqref{mod_EL} and collecting terms, we get that
\eqref{eqn2.13} holds for $r$.
\end{proof}
\begin{prop}\label{monot_mod_CS}
 Let $r\in C^2(0,1]$ be a solution of \eqref{mod_EL}. Then
$\hat{T}(r(\cdot))\in C^1(0,1]$ and
\begin{equation}\label{mod_CS_der}
 \td{}{R}\hat{T}(r(R))=(n-1)\kappa\,\frac{\deru{r}}{r}\left(1-\frac{
(\deru{r})^{n-1}}{(r/R)^{n-1}}\right).
\end{equation}
In particular, for a cavitating solution $r$, the function $\hat{T}(r(\cdot))$ 
is monotone increasing in $(0,1]$. Moreover, if $r(0)=0$, then $r(R)=\lambda R$ 
for $R\in[0,1]$.
\end{prop}
\begin{proof}
It follows from \eqref{Cauchy_stress} and \eqref{mod_CS} that
 \[
\hat{T}(r(R))=T(r(R))+(n-1)\kappa\left(\frac{1}{n}+\ln(r/R)\right). 
\]
Moreover, from \cite[Eq. 6.8]{Ba82} we have that for
\eqref{ogden-mat}:
\[
\td{}{R}T(r(R))=(n-1)\kappa\frac{R^{n-1}}{r^n}
\left(\left(\frac{r}{R}\right)^n-(\deru{r})^n\right).
\]
Hence
\begin{eqnarray*}
 \td{}{R}\hat{T}(r(R))&=&\td{}{R}T(r(R))+(n-1)\kappa\td{}{R}\ln(r/R),\\
 &=&(n-1)\kappa\left[\frac{R^{n-1}}{r^n}\left(\left(\frac{r}{R}\right)^n-
(\deru{r})^n\right)+\frac{1}{r}\left(\deru{r}-\frac{r}{R}\right)\right],\\
&=&(n-1)\kappa\frac{\deru{r}}{r}\left(1-\frac{(\deru{r})^{n-1}}
{(r/R)^{n-1}}\right),
\end{eqnarray*}
from which \eqref{mod_CS_der} follows.

The statement for the case in which $r(0)=0$
follows from $r$ being a solution of \eqref{eqn2.13}, assumption H2, and 
arguing as 
in \cite[Thm. 6.6]{Ba82}.
\end{proof}

Corresponding to the function \eqref{mod_Phi} we define
\begin{equation}\label{mod_bi_T}
 \hat{T}(\nu,v)=v^{1-n}\hat{\Phi}_{,1}(\nu,v,\ldots,v)=\kappa\left(
\frac{\nu}{v}\right)^{n-1}\!\!+\deru{h}(\nu v^{n-1})
+(n-1)\kappa\left(\frac{1}{n}+\ln(v)\right).
\end{equation}
For fixed $v>0$ we have that $\hat{T}(\nu,v)\searrow-\infty$ as $\nu\searrow0$
and $\hat{T}(\nu,v)\nearrow\infty$ as $\nu\nearrow\infty$. These together with
$\hat{T}_\nu(\nu,v)>0$ implies that the equation $\hat{T}(\nu,v)=C$ has a
unique solution $\hat{\nu}(C,v)>0$ for any $C\in\Real$. Let
\[
g(v)=\hat{T}(v,v)=\kappa+\deru{h}(v^n)
+(n-1)\kappa\left(\frac{1}{n}+\ln(v)\right).
\]
We note that $g(v)\searrow-\infty$ as $v\searrow0$,
$g(v)\nearrow\infty$ as $v\nearrow\infty$, and $\deru{g}(v)>0$. Thus the
equation $g(v)=C$ has a unique solution $\bar{v}(C)$ for any $C\in\Real$.

We now show that for ``small'' $\lambda$ the minimizers of \eqref{functIm} are
homogeneous, i.e., equal to $\lambda R$, and for $\lambda$ sufficiently large
they must be cavitating, i.e., with $r(0)>0$. The proof of the following 
proposition is an adaptation of the one in \cite{Ba82}, to the stored energy 
function \eqref{mod_Phi}.
\begin{prop}
 Let $r$ be any minimizer of $\hat{I}$ over \eqref{alambda} and assume that
\eqref{hcond} and \eqref{const_H1} hold. Then
\begin{enumerate}
 \item
 for $\lambda<\bar{\lambda}$ we must have that $r(R)=\lambda R$ where
$\bar{\lambda}$ is the solution of $\hat{T}(\bar{\lambda},\bar{\lambda})
=0$.
\item
For $\lambda$ sufficiently large we must have that $r(0)>0$.
\end{enumerate}
\end{prop}
\begin{proof}
 That $\bar{\lambda}$ exists and is unique follows from our previous comments.
Let $\lambda<\bar{\lambda}$ and $r$ be the corresponding minimizer of
\eqref{functIm} over $\mc{A}_\lambda$. Assume that $r(0)>0$. Then since
$r(1)=\lambda$ and $r(R)/R\To\infty$ as $R\searrow0$, we have that
$r(R_0)/R_0=\bar{\lambda}$ for some $R_0\in(0,1)$. Since $\hat{T}_\nu>0$ and
$\deru{r}(R_0)<r(R_0)/R_0=\bar{\lambda}$, we have that
\[
 0=\hat{T}(\bar{\lambda},\bar{\lambda})>\hat{T}(\deru{r}(R_0),\bar{\lambda}) =
 \hat{T}(r(R_0)).
\]
But from \eqref{mod_CS_der} we have that $\hat{T}(r(\cdot))$ is increasing and
since $\limite{R}{0}\hat{T}(r(R))=0$ we must have $\hat{T}(r(R))\ge0$ for
$R\in(0,1]$
which contradicts the above inequality. Hence $r(0)=0$ and from the last 
part of Proposition \ref{monot_mod_CS} we get that $r(R)=\lambda R$.

For the second part of the proof, we define $\hat{r}(R)=\sqrt[n]{dR^n+1-d}$
where $d\in(0,1)$. Is easy to check that $\deru{\hat{r}}(\hat{r}/R)^{n-1}=d$.
If we let $u(R)=\lambda \hat{r}(R)$, then $u\in\mc{A}_\lambda$. It 
follows now that
\begin{eqnarray*}
 \frac{\hat{I}(u)-\hat{I}(\lambda R)}{\lambda^n}&=&\int_0^1R^{n-1}\Big[
\frac{\kappa}{n}((\deru{\hat{r}})^n-1)+(n-1)\kappa(d-1)\ln(\lambda)\\
&&+\frac{n-1}{n}\,\kappa(d-1)+(n-1)\kappa d\ln(\hat{r}/R)+ 
(h(d\lambda^n)-h(\lambda^n))/\lambda^n\Big]\,
\dif R.
\end{eqnarray*}
Since $h(\cdot)$ is convex we get that $h(\lambda^n)\ge h(d\lambda^n)+(1-d)
\lambda^n\deru{h}(d\lambda^n)$ which implies
\[
 \frac{h(d\lambda^n)-h(\lambda^n)}{\lambda^n}\le(d-1)\deru{h}(d\lambda^n).
\]
Thus
\begin{eqnarray*}
 \frac{\hat{I}(u)-\hat{I}(\lambda R)}{\lambda^n}&\le&\int_0^1R^{n-1}\Big[
\frac{\kappa}{n}((\deru{\hat{r}})^n-1)+(n-1)\kappa(d-1)\ln(\lambda)\\
&&+\frac{n-1}{n}\,\kappa(d-1)+(n-1)\kappa 
d\ln(\hat{r}/R)+(d-1)\deru{h}(d\lambda^n)\Big]\,
\dif R.
\end{eqnarray*}
Since $d\in(0,1)$, the right hand side of this inequality is negative for
$\lambda$ large enough. Thus $\hat{I}(u)<\hat{I}(\lambda R)$ for $\lambda$
large enough. Hence the minimizer $r$ must have $r(0)>0$.
\end{proof}

If we let $\omega=R/r(R)$, then
\[
 \td{\omega}{R}=\dfrac{1-\omega\deru{r}}{r}.
\]
We now express the modified Cauchy stress $\hat{T}(r(R))$ as a function of 
$\omega$. In
reference to \eqref{mod_bi_T} we have that the equation
$\hat{T}(\deru{r},\omega^{-1})=T$ has a unique solution
$\deru{r}= \hat{\nu}(T,\omega^{-1})$. Moreover, the function
$\hat{\nu}$, as a function of $(T,\omega)$, can be extended to a bounded 
function for
$(T,\omega)\in[0,T_0]\times[0,\omega_0]$ for some $T_0>0$ and $\omega_0>0$. Also
\[
\pd{\hat{\nu}}{T}=\dfrac{\omega^{n-1}}{(n-1)\kappa\hat{\nu}^{n-1}\omega^{2(n-1)}
+\derd{h}(\hat{\nu}\omega^{-(n-1)})},
\]
which can also be extended to a bounded function in 
$[0,T_0]\times[0,\omega_0]$. Using Proposition
\ref{monot_mod_CS} we now get that $\hat{T}(\omega)$ is a solution of the
initial value problem
\begin{equation}\label{IVP_T}
\left\{\begin{array}{rcl}\displaystyle\td{\hat{T}}{\omega}(\omega)
&=&(n-1)\kappa
\sum_{k=0}^{n-2}\omega^k\hat{\nu}(\hat{T}(\omega),\omega^{-1})^{k+1},\\&&\\
        \hat{T}(0)&=&0.
       \end{array}\right.
\end{equation}
By the boundedness properties quoted above, the solution of this initial value
problem exists and is unique. Using this, the existence of a critical boundary 
displacement $\lambda_c$ can be established and the uniqueness of solutions 
for $\lambda>\lambda_c$ follows from a rescaling argument. The details of the 
previous argument leading to the initial value problem \eqref{IVP_T}, as 
well as the proof of the following proposition, are as in \cite{Ba82}.
\begin{prop}\label{lambda_crit}
Let $r_c$ be a cavitating solution of \eqref{mod_EL} satisfying 
\eqref{mod_BC_0},
and assume
that $\hat{\Phi}_{,1}(1,1,\ldots,1)=0$. Then $r_c$ can be
extended into $(0,\infty)$ with
\[
\deru{r_c}(R)<\frac{r_c(R)}{R},\quad R\in(0,\infty).
\]
Moreover, the function $r_c$ so extended is unique (does not depend on 
$r(1)$)
and
there exists $\lambda_c>1$ such that
\[
\lambda_c=\limite{R}{\infty}\deru{r_c}(R)=\limite{R}{\infty}\frac{r_c(R)}{R}.
\]
If $r_\lambda$ is a solution of \eqref{mod_EL} satisfying 
\eqref{mod_BC_0} and $r(1)=\lambda$ with $\lambda>\lambda_c$, then 
$r_\lambda(R)=r_c(\alpha R)/\alpha$ where $\alpha$ is the unique solution of 
$r_c(\alpha)/\alpha=\lambda$.
\end{prop}

It follows now that
\[
\limite{R}{\infty}\hat{T}(r_c(R))=\lambda_c^{1-n}\hat{\Phi}_{,1}(\lambda_c,
\lambda_c,\ldots,\lambda_c).
\]
Combining this with Proposition \ref{monot_mod_CS} we get that
\begin{equation}\label{lamb_crit_eqn}
\lambda_c^{1-n}\hat{\Phi}_{,1}(\lambda_c,\lambda_c,\ldots,\lambda_c)= 
(n-1)\kappa\,\int_0^\infty
\frac{\deru{r_c}(R)}{r_c(R)}\left(1-\frac{(\deru{r_c}(R))^{n-1}}{(r_c(R)/R)^{n-1
}}\right)\,\dif R.
\end{equation}
\section{Approximations by punctured balls}\label{sec:punct}
We now consider the problem over the punctured ball:
\[
 \mc{B}_\eps=\set{\ts{x}\in\Real^n\,:\,\eps<\abs{\ts{x}}<1},
\]
with $\eps\in(0,1)$. Thus we look at the problem of minimizing
\begin{equation}\label{functIm_eps}
\hat{I}_\eps(r)=\int_\eps^1R^{n-1}\hat{\Phi}
\left(\deru{r}(R),\frac{r(R)}{R},\ldots,
\frac{r(R)}{R}\right)\,\dif R- \frac{\kappa 
(n-1)}{n}\,\lambda^n\ln\lambda,
\end{equation}
over the set
\begin{equation}\label{alambda_eps}
\mathcal{A}_\lambda^\eps=\set{r\in
W^{1,1}(\eps,1)\,:\,r(1)=\lambda,\,\deru{r}(R)>0\mbox{~a.e.
for~}R\in(\eps,1),\, r(\eps)\ge0}.
\end{equation}
To state our next result we shall need the following lemma.
\begin{lem}
 Let $\bar{\lambda}=1$ be the unique solution of $\hat{\Phi}_{,1}(\bar{\lambda},
\bar{\lambda},\ldots,\bar{\lambda})=0$. Then 
\[
 \hat{\Phi}(v_1,v_2,\ldots,v_n)>\hat{\Phi}(1,1,\ldots,1),
\]
whenever $v_i\ne1$ for some $i$.
\end{lem}
\begin{proof}
 From \eqref{mod_Phi} we have that
 \[
  \hat{\Phi}(v_1,v_2,\ldots,v_n)=g(v_1,v_2\cdots v_n),
 \]
where
\[
 g(x,y)=\frac{\kappa}{n}\,x^n+h(xy)
 +\kappa xy((n-1)/n+\ln(y)),\quad x>0,\,\,y>0.
\]
The critical points of $g$ are given by the solutions of the system
\[
 \left\{\begin{array}{rcl}\kappa x^{n-1}+y\deru{h}(xy)+\kappa 
y((n-1)/n+\ln(y))&=&0,\\
         x\deru{h}(xy)+\kappa x(1+(n-1)/n+\ln(y))&=&0.
        \end{array}\right.
\]
This system has a unique solution given by the equations
\[
 y=x^{n-1},\quad \deru{h}(xy)=-\kappa(1+(n-1)/n+\ln(y)).
\]
The condition $\bar{\lambda}=1$ is the only 
solution of 
$\hat{\Phi}_{,1}(\bar{\lambda},\bar{\lambda},\ldots,\bar{\lambda})=0$ implies 
that the only solution 
of these equations is $x=y=1$. Moreover since $g_{xx}(x,y)>0$ and
\[
 g_{xx}(1,1)g_{yy}(1,1)-g_{xy}(1,1)^2>0,
\]
we have that $(1,1)$ is a strict local minimum for $g$. Since $g(x,y)\To\infty$ 
as any 
of its arguments tend to zero or infinity, this minimum is global. Thus 
whenever 
$v_i\ne1$ for some $i$, we have
\[
\hat{\Phi}(v_1,v_2,\ldots,v_n)=g(v_1,v_2\cdots v_n)>g(1,1)=\hat{\Phi}(1,1,\ldots
,1).
 \]
\end{proof}
With slight modifications of the proofs of Theorems \ref{min_existence} and
\ref{mod_EL1}, we get that the following result
holds for minimizers of \eqref{functIm_eps} over \eqref{alambda_eps}. (See also 
\cite{Si86a}.)
\begin{thm}\label{min_existence_EL_eps}
Let $\bar{\lambda}=1$ be the unique solution of $\hat{\Phi}_{,1}(\bar{\lambda},
\bar{\lambda},\ldots,\bar{\lambda})=0$. Then the functional 
$\hat{I}_\eps$ has a unique global minimizer over the set
$\mathcal{A}_\lambda^\eps$. Moreover, there exists a
$\delta(\eps)>0$ such that if $r_\eps$ is a global minimizer with
$\lambda\in(1-\delta(\eps),\infty)$, then $r_\eps\in
C^2([\eps,1])$ is a solution of \eqref{mod_EL} over $(\eps,1)$,
and satisfies:
\begin{enumerate}
\item
$\deru{r_\eps}(R)>0$ for $R\in[\eps,1]$,
\item
$r_\eps(\eps)>0$,
\item
$\hat{T}(r_\eps(\eps))=0$.
\end{enumerate}
\end{thm}
We also have (see \cite{Si86a}):
\begin{prop}\label{ep_convergence}
Let $r_\eps$ be the unique global minimizer of $\hat{I}_\eps$ over
$\mathcal{A}_\lambda^\eps$ and let $\lambda_{c}$ be as in
Proposition \ref{lambda_crit}. Then
\begin{enumerate}
\item
for $\lambda\le\lambda_{c}$, we have that
\[
\limite{\eps}{0}\,\sup_{R\in[\eps,1]}\abs{r_\eps(R)-\lambda R}=0,
\]
\item
if $\lambda>\lambda_{c}$, then we have that
\[
\limite{\eps}{0}\,\sup_{R\in[\eps,1]}\abs{r_\eps(R)-r_{\lambda}(R)}=0,
\]
where $r_{\lambda}$ is the cavitating minimizer of $\hat{I}$ over $A_\lambda$.
\end{enumerate}
\end{prop}

We recall (cf. \cite{Si86a}) that the change of variables
\begin{equation}\label{cvar2}
\me^s=R,\quad v(s)=\frac{r(R)}{R},
\end{equation}
transforms equation \eqref{mod_EL} into the autonomous equation:
\begin{eqnarray}
\td{}{s}\hat{\Phi}_{,1}(\dot{v}(s)+v(s),v(s),\ldots,v(s))&=&
(n-1)\left(\hat{\Phi}_{,2}(\dot{
v } (s)+v(s),v(s),\ldots,v(s))\right.\nonumber\\&&-
\left.\hat{\Phi}_{,1}(\dot{v}(s)+v(s),v(s),\ldots,v(s))\right),\label{tran-de}
\end{eqnarray}
where $\dot{v}(s)=\dif v(s)/\dif s$. Now a phase plane analysis of this 
equation 
in the $(v,\dot{v})$ plane, based on the time map \cite[Eq. (2.19)]{Si86a}, 
the monotonicity of the Cauchy stress $\hat{T}(r(\cdot))$ along solutions (cf. 
Proposition \ref{monot_mod_CS}), and the continuous dependence on initial data 
for solutions of \eqref{tran-de}, shows that the following results concerning 
the convergence of the strains corresponding to the solutions $r_\eps$ in 
Proposition \ref{ep_convergence} hold (here $v_\eps$ is the solution of 
\eqref{tran-de} corresponding to $r_\eps$):
\begin{enumerate}
 \item
 For $\lambda>\lambda_{c}$, the strains $(\dot{v}_\eps+v_\eps,v_\eps)$ 
converge as $\eps\To0$ to the strains 
$(\dot{v}_{\lambda}+v_{\lambda},v_{\lambda})$ corresponding to the cavitating 
solution $r_{\lambda}$.
\item
For $\bar{\lambda}<\lambda<\lambda_{c}$, the strains 
$(\dot{v}_\eps+v_\eps,v_\eps)$
converge as $\eps\To0$ to the strains corresponding to the nonhomogeneous 
solution
$(v,\dot{v})$ emanating from $(\lambda,\lambda)$ and with $\dot{v}<0$. The
convergence is such that $(\dot{v}_\eps,v_\eps)$ spends most of the time 
(in the sense of \cite[Eq. (2.19)]{Si86a}) close to $(\lambda,\lambda)$ than to 
the rest of the curve corresponding to the boundary condition
$\hat{T}(r_\eps(\eps))=0$. Thus the strains $(\deru{r}_\eps,r_\eps/R)$ develop a
sharp boundary layer close to $R=\eps$ while away from this point they each tend
to $\lambda$.
\item
For $\lambda<\bar{\lambda}$, we have the same conclusions as in 2) above but 
with
$\dot{v}>0$, i.e., with $\deru{r}_\eps>r_\eps/s$.
\end{enumerate}

\section{Numerical results}\label{sec:num}
In this section we present some numerical results that highlight the 
convergence results in Section \ref{sec:punct} over punctured balls. 
We employ two numerical schemes: a descent method based on a 
gradient flow iteration (cf. \cite{Neu1997}) for the minimization of a discrete 
version of \eqref{functIm_eps}; and a shooting method (from $R=1$ to 
$R=\eps$) to solve the boundary value problem for \eqref{mod_EL} over 
$(\eps,1)$ with boundary conditions $\hat{T}(r_\eps(\eps))=0$ and 
$r_\eps(1)=\lambda$. The gradient flow iteration works as a 
\textit{predictor} for the shooting method which in turn plays the role of a 
\textit{corrector}. The use of 
adaptive ode solvers in the shooting method allows for a more precise 
computation of the equilibrium states, especially near $R=\eps$ where the 
strains corresponding to the punctured ball solutions tend to 
develop sharp boundary layers.
\begin{example}
For the stored energy function \eqref{mod_Phi} (or \eqref{ogden-mat}), we take
\[
h(d)=C\,d^{\gamma}+D\,d^{-\delta},
\]
where $C,D\ge0$ and $\gamma,\delta>0$. The reference configuration
is stress free, that is 
$\hat{\Phi}_{,1}(1,\ldots,1)=\hat{\Phi}_{,2}(1,\ldots,1)=0$, 
provided:
\[
D=\dfrac{(1+\frac{n-1}{n})\kappa+C\gamma}{\delta}.
\]
For the computations we used the following values for the different parameters:
\[
 n=3,\quad\kappa=1,\quad C=1,\quad\gamma=2,\quad\delta=2.
\]
For these values, the critical boundary displacement is  
$\lambda_c\approx1.0258$ (cf. \cite{NeSi2009}). For $\eps=0.3, 
0.2, 10^{-4}$ and $\lambda=1.05$ 
(case $\lambda>\lambda_c$) we show in Figure \ref{fig:1} the 
computed solutions 
$r_\eps$ and the modified Cauchy stress functions $\hat{T}(r_\eps(\cdot))$, the 
former 
converging very nicely to a cavitating 
solution while the later to a well defined increasing function vanishing at 
$R=0$. The cavity size for the computed solution with $\eps=10^{-4}$ is 
$0.44184$ approximately with modified energy of $1.2774$. The affine 
deformation in this case has energy of $1.2888$. 

For $\lambda=1.01$ which corresponds to the 
case $\bar{\lambda}<\lambda<\lambda_c$, as $\bar{\lambda}=1$, we show in Figure 
\ref{fig:4} the computed $r_\eps$ and $\hat{T}(r_\eps(\cdot))$. The convergence 
is now to the affine deformation $r_h(R)=1.01R$ with energy of $1.2733$. The 
corresponding Cauchy stress functions show sharp boundary layers at $R=\eps$ 
while converging pointwise to a positive constant function.

The other calculation we show is for 
$\lambda=0.95$ (case $\lambda<\bar{\lambda}=1$) with the same values of $\eps$. 
The results are presented in Figure \ref{fig:3} where we can clearly see the 
convergence of the $r_\eps$ to the affine deformation $r_h(R)=0.95R$ 
with energy of $1.3625$ (Figure \ref{fig:rc3}). The $r_\eps$'s in this figure 
are concave functions corresponding to the case where $\dot{v}>0$ in 
\eqref{tran-de}. On the other hand in Figure \ref{fig:Tc} we see the 
corresponding Cauchy stress functions converging pointwise, with a sharp 
boundary layer at $R=\eps$, to a negative constant function.
\end{example}
\begin{example}
 In this example we study the so called \textit{incompressible limit}
by considering a sequence of compressible problems formally approaching an 
incompressible one. In particular we consider functions $h(\cdot)$ in 
\eqref{mod_Phi} given by
\[
h(d)=C\,\left(d-1-\frac{1}{C}\right)^{2}+D\,d^{-\delta},
\]
where $C,D\ge0$ and $\delta>0$. As $C\To\infty$ we formally approach the 
incompressible modified stored energy function given by:
\[
 \hat{\Phi}^{inc}(v_2,\ldots,v_n)=\frac{\kappa}{n}\,(v_2\cdots v_n)^{-n}+D
 +\kappa \left(\frac{n-1}{n}+\ln(v_2\cdots v_n)\right),
\]
For the computations we used the following:
\[
 n=3,\quad\kappa=3,\quad D=1.5,\quad\delta=2,
\]
with $\lambda=1.05$. In Figure \ref{fig:5} we show in solid the solution of the 
incompressible problem which is given by $r_{inc}(R)=\sqrt[3]{R^3+ 
\lambda^3-1}$, together with the computed minimizers of the modified 
compressible problems \eqref{functIm_eps} with $\eps=0.005$ and $C=20,40$ 
(dashed and dotted respectively), which are clearly seen getting close to 
$r_{inc}$. We computed as well solutions of the modified compressible problems 
for additional values of $C$ together with their modified energies. The results 
are shown in Table \ref{tab:1}. The energy of $r_{inc}$, computed using 
$\hat{\Phi}^{inc}$ above, is given approximately by $1.53013$. Thus we see as 
well a nice convergence of the energies of the modified compressible problems 
in 
the incompressible limit.
\end{example}
\begin{table}[tbhp]
{\footnotesize
\caption{Energies for the modified compressible problems in 
the incompressible limit case.}\label{tab:1}
\begin{center}
\begin{tabular}{|c|c||c|c|}\hline
$C$&$\hat{I}_{\eps}(r_\eps)$&$C$&$\hat{I}_{\eps}(r_\eps)$\\\hline\hline
20&1.52298&160&1.52864\\\hline
40&1.52532&320&1.52936\\\hline
80&1.52735&640&1.52974\\\hline
\end{tabular}
\end{center}
}
\end{table}

\section{Concluding Remarks}
It is not difficult to check that the results of this paper can be generalized to stored energy functions of the form
\begin{equation}\label{SEnormn}
 W(\ts{F})=\frac{\kappa}{n}\norm{\ts{F}}^n+h(\det\ts{F})
 =\frac{\kappa}{n}(v_1^2+\cdots+v_n^2)^{\frac{n}{2}}+h(v_1\cdots v_n).
\end{equation}
In fact, an analysis for this stored energy function similar to the one leading 
to Proposition \ref{Tasint}, shows that $T(v)$ is now asymptotic to 
$-\kappa(n-1)^{n/2}\ln(v)$ as $v\To\infty$. Thus we are led to consider a 
modified functional of the form
\begin{eqnarray*}
\hat I(r)&=&\int_0^1  
R^{n-1}\Bigg[\frac{\kappa}{n}\left(\left[\deru{r}(R)^2+(n-1)\left(\frac{r}{R}\right)^2\right]^{\frac{n}{2}}- 
(n-1)^{\frac{n}{2}}\left(\frac{r}{R}\right)^n\right)\\&&+h\left(\delta 
(R)\right)
+\kappa(n-1)^{\frac{n}{2}}\delta(R)\left(\frac{1}{n}+ 
\ln\left(\frac{r}{R}\right)\right)
\Bigg]\,\dif R
-\frac{\kappa(n-1)^{\frac{n}{2}}}{n}\lambda^n\ln\lambda.
\end{eqnarray*}
As this functional can be characterised in terms of the original one plus 
suitable null Lagrangians, its Euler--Lagrange equation coincides with that of 
the original functional. The rest of the analysis in this paper should now follow through.

The radial incompressible case can be treated similarly to the 
compressible case studied in this paper. However, the incompressible case is 
more straightforward since a radial incompressible deformation of the form 
\eqref{eqn2.8} which also satisfies \eqref{bcond} is necessarily given 
by 
$$
r(R)=(R^n+(\lambda ^3 -1))^{\frac 1n},
$$
for $\lambda >1$.
On using this form, \cite[Prop. 5.1]{Ba82} shows that 
\begin{eqnarray*}
 &-\int_\lambda^b\frac{1}{v^n-1}\td{}{v}\Phi(v^{1-n},v,\ldots,v)\,\dif v+
 n\int_\lambda^b\frac{v^{n-1}}{(v^n-1)^2}\Phi(v^{1-n},v,\ldots,v)\,\dif v&\\&
 ~~~~~~~~~~~~~~~~~~~~~~~~~~~~~~~~~~~= 
\frac{1}{\lambda^n-1}\Phi(\lambda^{1-n},\lambda,\ldots,\lambda),&
\end{eqnarray*}
for any $b>\lambda$ (the case $b$ finite corresponds to integrating over a 
punctured ball in the reference configuration of internal radius $\left( 
\frac{\lambda^n-1}{b^n-1}\right)^{\frac 1n}$. As $b\To\infty$ (corresponding to 
the puncture closing up), the first term on the left of this 
equation is, up to a constant, the  
radial Cauchy stress (on the deformed puncture surface) whilst the second term is the $n$ times the energy of the deformed punctured ball. Taking the form of $\Phi$ in this 
incompressible case as $(\kappa/n)\sum_{i=1}^nv_i^n$ plus some 
constant, it is easy to obtain from the expression above that the growth in the 
radial Cauchy 
stress is once again asymptotically proportional to $\ln(b)$ as $b\To\infty$.

In generalising the techniques in this paper from radially symmetric 
deformations to none radial ones, one approach (cf. \cite{SiSpTi2006}) is to  
restrict attention to deformations for which the distributional determinant of 
the deformation satisfies:
$$
\mbox{Det}(\nabla \bu)=(\det\nabla\bu) \mathcal{L}^n +V_{\ts{u}}\delta_{\ts{0}},
$$
where $\delta_{\ts{0}}$ is the Dirac measure supported at the origin and 
$V_{\ts{u}}$ is the volume of the cavity formed by the deformation 
$\bu$ at the origin. From \cite[Proposition 3.6]{SiSp2010b} we get that in 
the case $n=3$,
\begin{equation}\label{Ineq}
 \int_{\B_\eps}\norm{\nabla\ts{u}}^3\,\dif\ts{x}\ge 
\int_{\B_\eps}\norm{\nabla\ts{u}^{\mbox{rad}}}^3\,\dif\ts{x}\ge 
-2^{\frac{3}{2}}\,\omega_3\,
\tilde{r}^3(\eps)\ln(\eps),
\end{equation}
where $\omega_3=4\pi$. Here $\ts{u}^{\mbox{rad}}$ is the radial symmetrization 
of $\ts{u}$ and is 
given by \eqref{eqn2.8} where $r$ is replaced by $\tilde{r}$ which in turn is 
given by
\[ 
\frac{4\pi}{3}\,\tilde{r}^3(R)=\frac{4\pi}{3}\,\lambda^3-\int_{\B_R}
\det(\nabla\ts
{ u } )\, \dif\ts{x}.
\]
(The inequality \eqref{Ineq} holds provided $\deru{\tilde{r}}(R)\le 
\tilde{r}(R)/R$ for all $R\in[\eps,1]$. If this condition is not satisfied, then 
the symmetrisation $\tilde{r}$ has to be modified as in  \cite{SiSpTi2006} in 
order for \eqref{Ineq} to hold.) Thus, it should 
follow from 
\eqref{Ineq} that the total energy due to the deformation $\ts{u}$ blows up 
at least like $-\ln(\eps)$ as $\eps\To0^+$ if $V_{\ts{u}}>0$. Thus, in 
generalizing 
our results to the non--radial case with the stored energy function 
\eqref{SEnormn}, we are led to consider a modified 
energy functional given by
\[
\hat{E}(\ts{u})=\lim_{\eps\rightarrow 0}\left[\int_{\B_\eps} W(\nabla 
\bu)\,\dif\ts{x} 
+2^{\frac{3}{2}}\kappa\,V_{\bu} \ln (\eps)\right].
\]
It may now follow from the approach in \cite{SiSp2010b}  
that, for each $\eps>0$, the minimizer of the functional in 
brackets above (over $\B_\eps$) must be radial. Under 
suitable hypotheses, it may then follow that the minimizer of $\hat{E}$ is 
radial and so the results of the current paper would then be applicable. We 
shall 
pursue these ideas elsewhere.
\pagebreak

\pagebreak
\begin{figure}
  \begin{tabular}{cc}
  \subfloat[$r_\eps$]
{\label{fig:rc1}\includegraphics[width=0.5\textwidth]{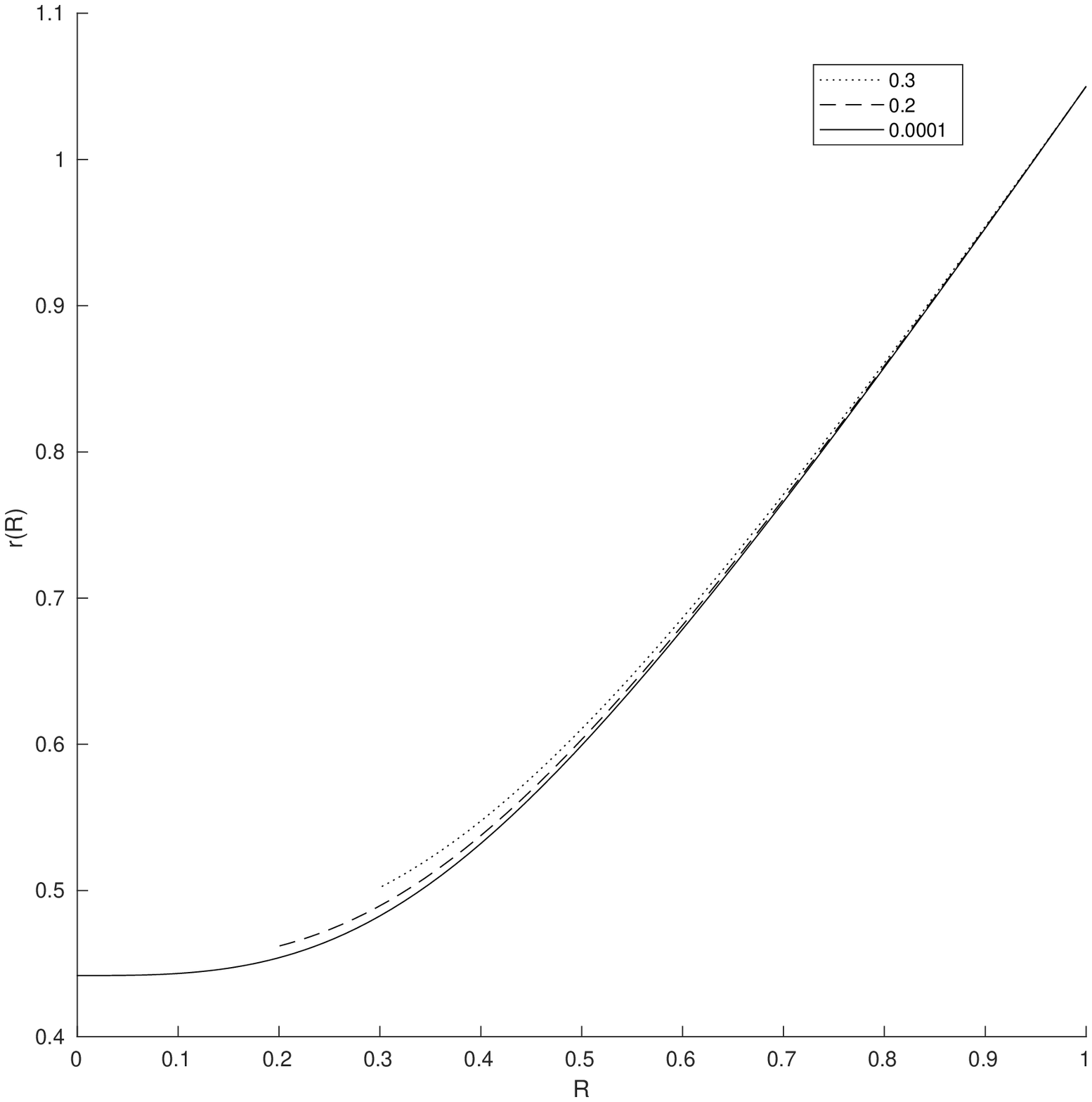}}&
  \subfloat[$\hat{T}(r_\eps(\cdot))$]
{\label{fig:Tc2}{\includegraphics[width=0.42\textwidth]{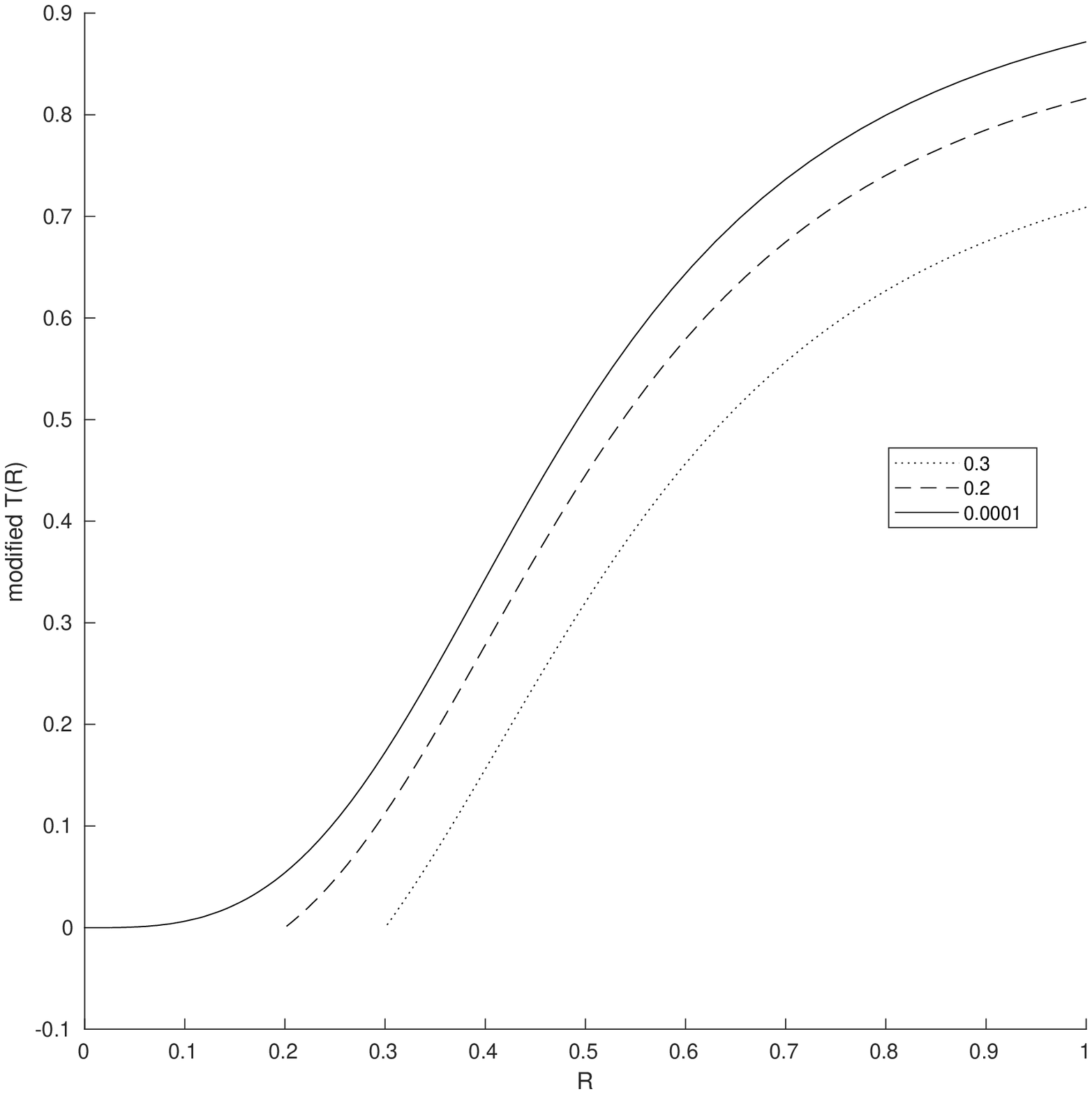}}}
\end{tabular}
\caption{Computed minimizers of $\hat{I}_\eps$ when $\lambda=1.05$ and
$\eps=0.3, 0.2, 10^{-4}$.}\label{fig:1}
\end{figure}
\begin{figure}
  \begin{tabular}{cc}
  \subfloat[$r_\eps$]
{\label{fig:rc4}\includegraphics[width=0.5\textwidth]{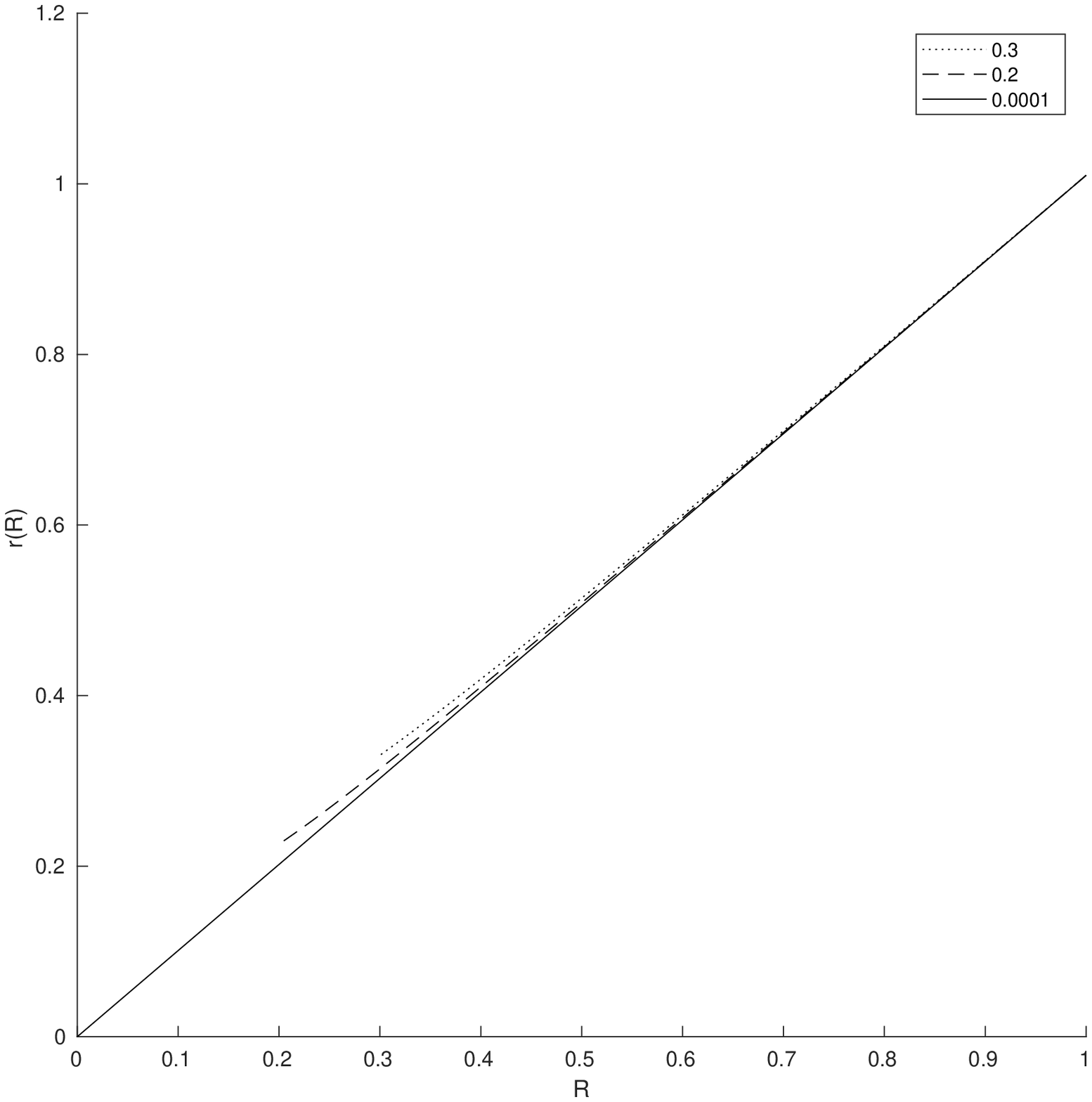}}&
  \subfloat[$\hat{T}(r_\eps(\cdot))$]
{\label{fig:Tc4}{\includegraphics[width=0.42\textwidth]{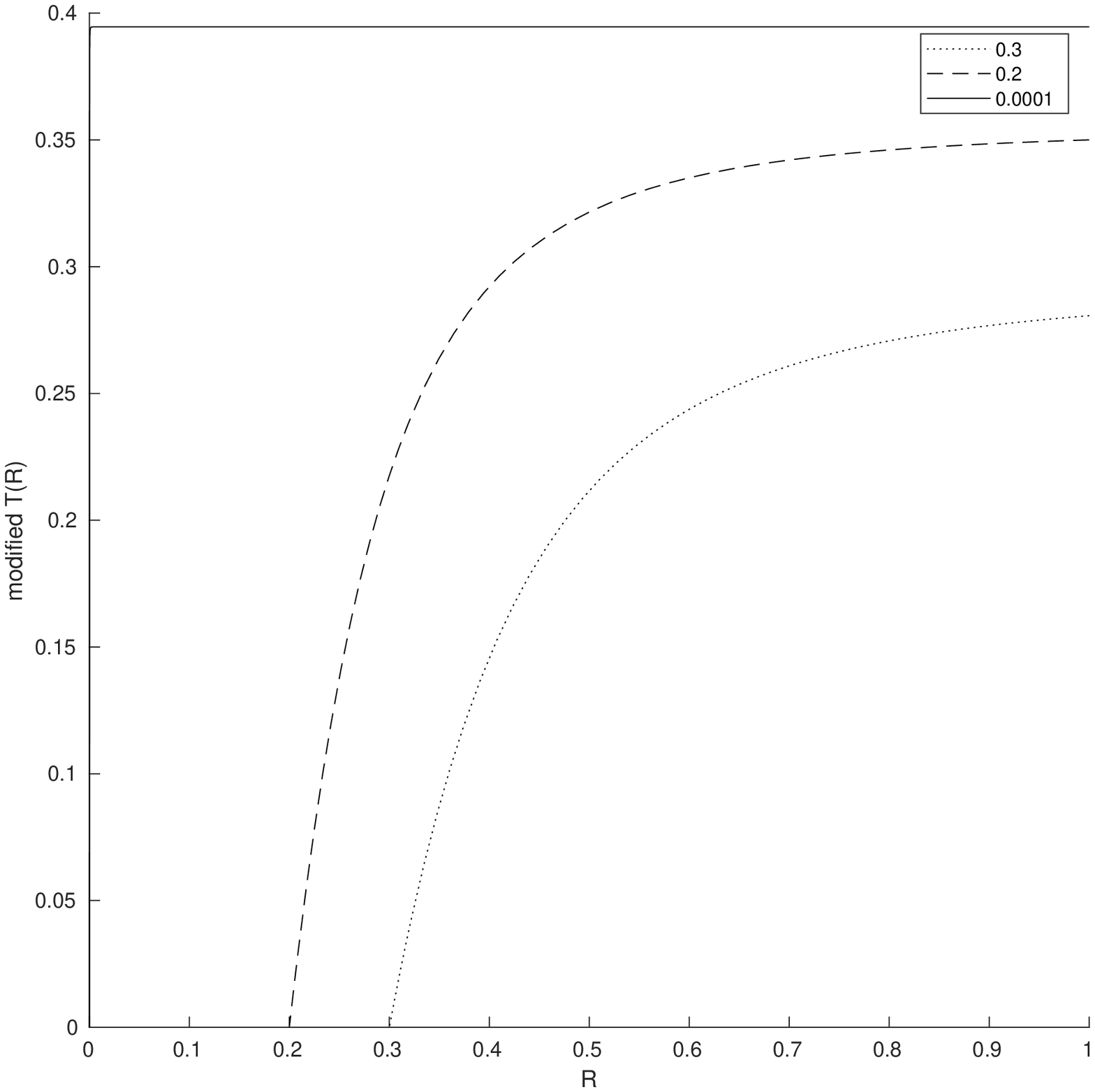}}}
\end{tabular}
\caption{Computed minimizers $r_\eps$ and modified Cauchy stress
functions $\hat{T}(r_\eps(\cdot))$ for $\hat{I}_\eps$ when $\lambda=1.01$ and
$\eps=0.2, 0.1, 0.05, 10^{-4}$.}\label{fig:4}
\end{figure}

\begin{figure}
  \begin{tabular}{cc}
  \subfloat[$r_\eps$]
{\label{fig:rc3}\includegraphics[width=0.5\textwidth]{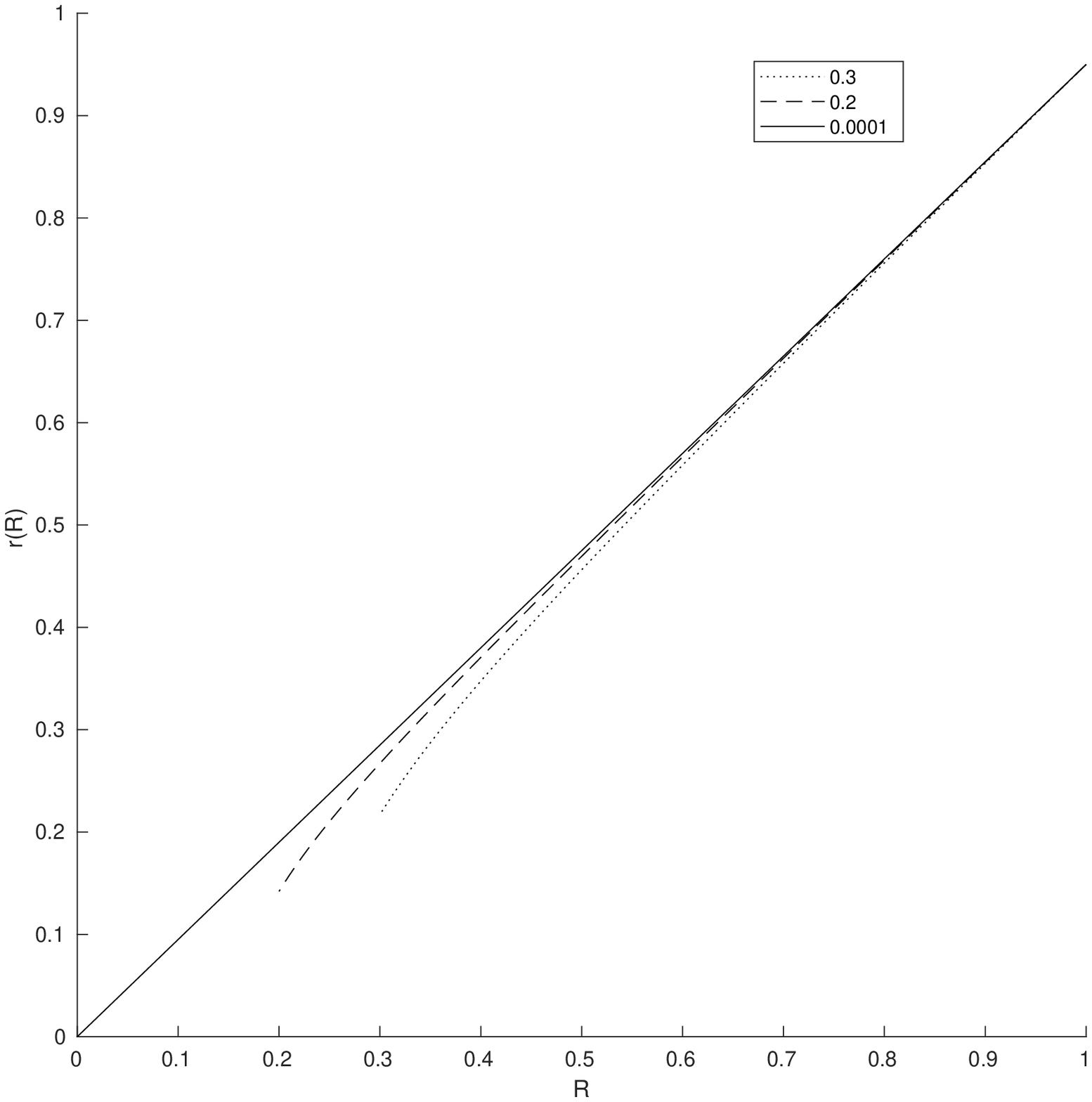}}&
  \subfloat[$\hat{T}(r_\eps(\cdot))$]
{\label{fig:Tc}{\includegraphics[width=0.42\textwidth]{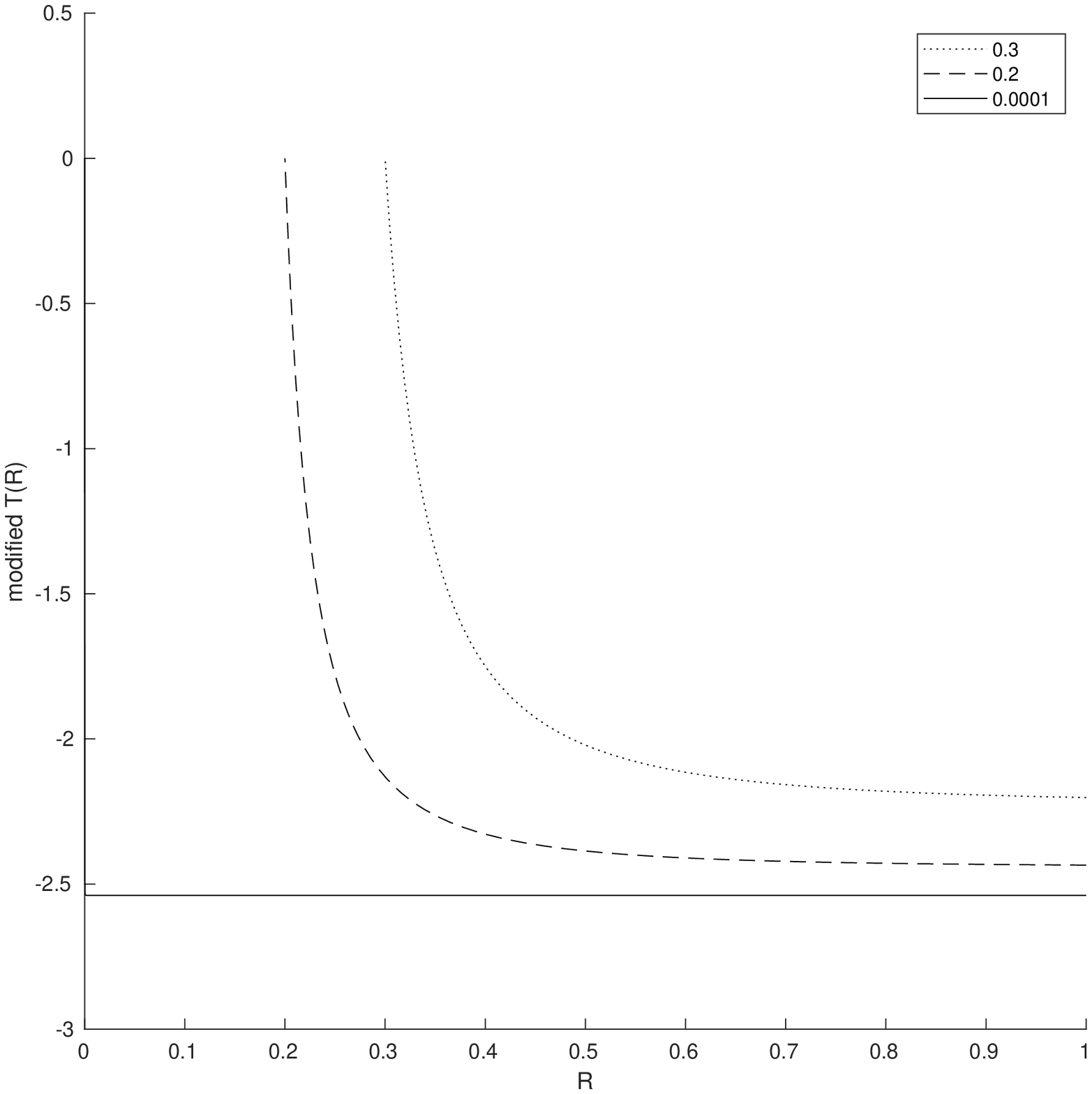}}}
\end{tabular}
\caption{Computed minimizers $r_\eps$ and modified Cauchy stress
functions $\hat{T}(r_\eps(\cdot))$ for $\hat{I}_\eps$ when $\lambda=0.95$ and
$\eps=0.2, 0.1, 0.05, 10^{-4}$.}\label{fig:3}
\end{figure}
\begin{figure}
\begin{center}
\scalebox{0.4}{\includegraphics{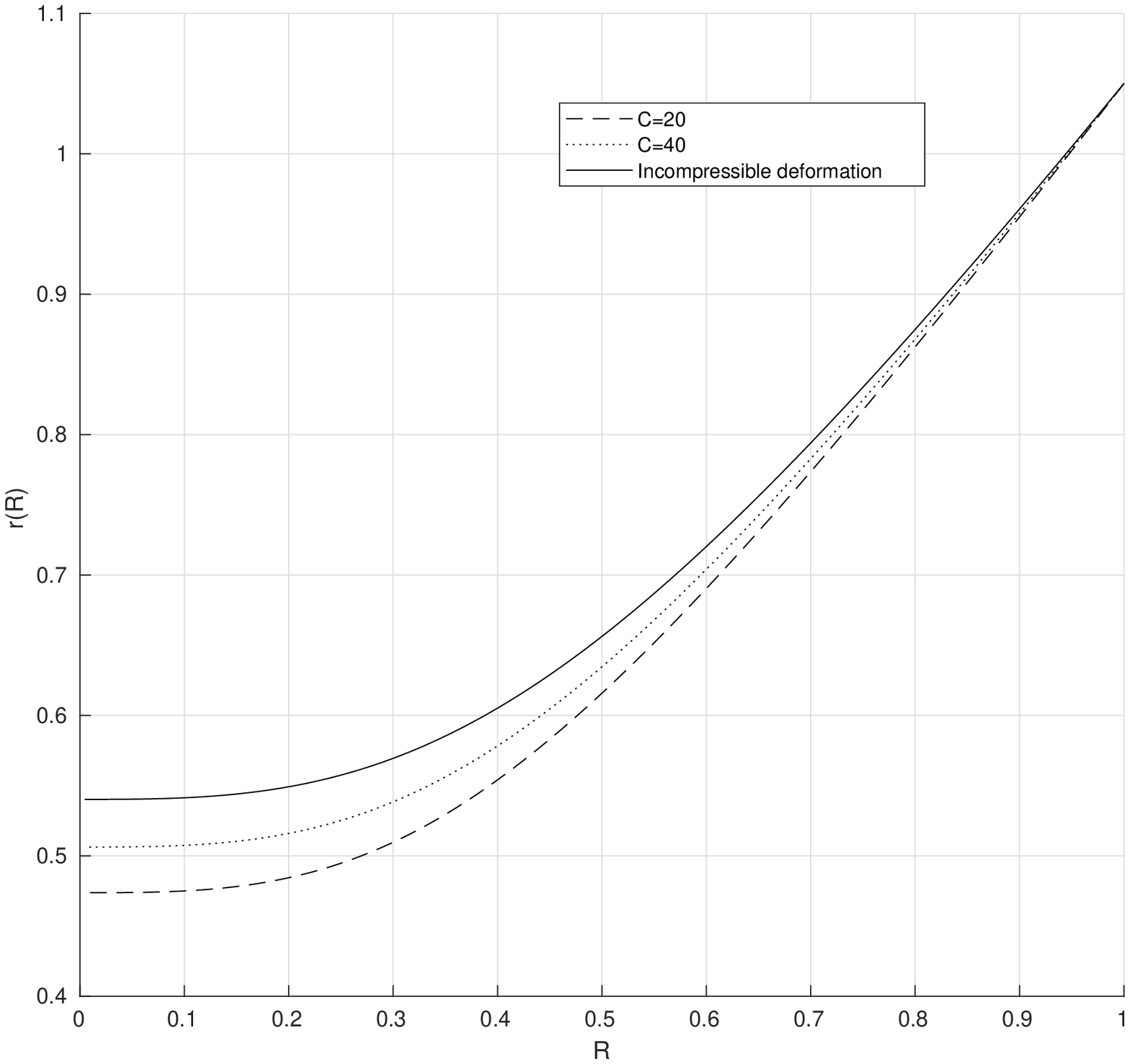}}
\end{center}
\caption{Minimizers of modified compressible problems approaching the 
incompressible deformation in the incompressible limit with $C$ increasing.}
\label{fig:5}
\end{figure}


\begin{thebibliography}{99}

\bibitem{Ba82}
    {\sc J. M. Ball}, \textit{Discontinuous Equilibrium Solutions and 
Cavitation in Nonlinear Elasticity},
    Phil. Trans. Royal Soc. London \textbf{A 306}, 557-611, 1982.
\bibitem{BeBrHe1994}
    {\sc F. B\'ethuel, H. Brezis, and F. H\'elein}, 
\textit{Ginzburg--Landau
vortices}, Progress in Nonlinear Differential Equations and Their Applications,
13. Birkh\:auser, Boston, 1994.
\bibitem{CaDu2019}
{\sc V. Ca\~nulef-Aguilar and D. Henao}, \textit{A lower bound for the 
void 
coalescence load in nonlinearly elastic solids}, Interfaces And Free 
Boundaries, Vol. 21, Issue 4, pp. 409--440, 2019.
\bibitem{Es1951}
{\sc J. D. Eshelby}, \textit{The force on an elastic singularity}, Proc. 
R. 
Soc. A244, 87–111, 1951.
\bibitem{GeLi58}
    {\sc  A.~N. Gent and P.~B. Lindley}, {\em Internal
rupture of bonded rubber cylinders in tension}, Proc. R. Soc. Lond.
A 249, pp. 195--205, 1958.
\bibitem{GoMi1996}
{\sc S. Govindjee and P. A. Mihalic}, {\em Computational methods for inverse 
finite elastostatics}, Comput. Methods Appl. Mech. Engrg., 136, 47--57, 1996.
\bibitem{Green1973}
{\sc A. E. Green}, \textit{On some general formulae in finite 
elastostatics}, 
Arch. Rational Mech. Anal. 50, pp. 73-80, 1973.
\bibitem{HMC}
    {\sc D. Henao and C. Mora-Corral}, {\em Invertibility and
weakcontinuity of the determinant for the modelling of cavitation and
fracture in nonlinear elasticity}, Arch. Rat. Mech. Anal., 197 (2010),
pp. 619--655.
\bibitem{HeSe2013}
    {\sc D. Henao and S. Serfaty}, \textit{Energy estimates and cavity
interaction for a critical-exponent cavitation model},
Communications on Pure and Applied Mathematics, Vol. LXVI, 1028--1101, 2013.
\bibitem{HeRo2018}
{\sc D. Henao and R. Rodiac}, {\em On the existence of minimizers for the 
neo-Hookean energy in the axisymmetric setting}, Discrete Contin. Dyn. Syst., 
38, 4509--4536, 2018.
\bibitem{HP}
    {\sc C. O. Horgan and D. A. Polignone}, {\em Cavitation in
nonlinearly elastic solids: A review}, Appl. Mech. Rev. 48 (1995),
pp. 471--485.
\bibitem{JSp2}
   {\sc  R. James and S. J. Spector}, {\em The formation of
filamentary voids in solids}, J. Mech. Phys. Solids 39 (1991),
pp. 783--813.
\bibitem{Me} 
{\sc F. Meynard}, {\em Existence and nonexistence results on the
radially symmetric cavitation problem}, Quart. Appl. Math.
50 (1992), pp. 201--226.
\bibitem{MuSp95} 
{\sc S. M{\"u}ller and S. J. Spector}, {\em An existence theory for
nonlinear elasticity that allows for cavitation}, Arch. Rat. Mech.
Anal., 131 (1995), pp. 1--66.
\bibitem{NeSi2009}
{\sc P. V. Negr\'on--Marrero and J. Sivaloganathan}, \textit{The 
Numerical 
Computation of the Critical Boundary Displacement for Radial Cavitation}, 
Mathematics and Mechanics of Solids, 14: 696-726, 2009.
\bibitem{Neu1997}
{\sc J. W. Neuberger}, Sobolev Gradients and Differential Equations, 
\textit{Lecture Notes in Math.} 1670, Springer, Berlin, 1997.
\bibitem{PeGu2015}
{\sc T. J. Pence and K. Gou}, {\em On compressible versions of the
incompressible neo-Hookean material}, Mathematics and Mechanics of Solids,
Vol. 20(2) 157--182, 2015. 
\bibitem{Si86a}
    {\sc J. Sivaloganathan}, \textit{Uniqueness of regular and singular 
equilibria for spherically symmetric problems of
    nonlinear elasticity}, Arch. Rat. Mech. Anal., 96, 97--136, 1986.
\bibitem{SiSp2000a}
    {\sc J. Sivaloganathan and S. J. Spector},
    {\em On the existence of minimizers with prescribed singular
points in nonlinear elasticity}, J. Elasticity, 59 (2000), pp.
83--113.
\bibitem{SiSp2000b}
{\sc J. Sivaloganathan and S. J. Spector},
{\em On the optimal location of singularities arising in variational problems 
of elasticity}, Journal of Elasticity, 58, 191--224, 2000.
\bibitem{SiSp2010b}
{\sc J. Sivaloganathan and S. J. Spector.},
\textit{On the symmetry of energy minimizing deformations in nonlinear 
elasticity II: compressible materials}, Arch. Rational Mech. Anal, 
Vol. 196, 395--431, 2010.
\bibitem{SiSpTi2006} 
{\sc J. Sivaloganathan, S. J. Spector, and V. Tilakraj},
\textit{The convergence of regularized
    minimizers for cavitation problems in  nonlinear elasticity}, SIAM
    J. Appl. Math, 66, 736-757, 2006.
\bibitem{St85}
{\sc C. A. Stuart}, \textit{Radially symmetric cavitation for hyperelastic
	materials}, Analyse non lineaire,
2, 33--66, 1985.
\bibitem{St93}
{\sc C. A. Stuart}, \textit{Estimating the critical radius for radially 
symmetric 
cavitation}, Quarterly of Applied Mathematics, Vol LI, 251-263, 1993.
\bibitem{VoGoYu1979}
{\sc S. K. Vodop’yanov, V. M. Gol’dshtein. and Yu. G. Reshetnyak}, 
\textit{The geometric properties of functions with generalized first 
derivatives}. Russian Math. Surveys, 34:19–74, 1979.
\end{thebibliography}
\end{document}